\documentclass[12pt,a4paper]{article}
\usepackage[latin1]{inputenc}
\usepackage{amsmath}
\usepackage{amsfonts}
\usepackage{amssymb}
\usepackage{amsthm}
\usepackage{cite}
\usepackage{a4wide}

\usepackage{tikz}

\usepackage{caption} 
\theoremstyle{plain}
\newtheorem{thm}{{Theorem}}[section]

\newtheorem{cor}[thm]{Corollary}

\newtheorem{pro}[thm]{Proposition}

\allowdisplaybreaks

\author{Isaac Owino Okoth \\ Department of Pure and Applied Mathematics \\
Maseno University, Kenya \\ ookoth@maseno.ac.ke \\\\
Stephan Wagner\footnote{Supported by the Knut and Alice Wallenberg Foundation, grant KAW 2017.0112.} \\ Department of Mathematics \\ Uppsala Universitet, Sweden \\
and \\  Department of Mathematical Sciences \\ Stellenbosch University, South Africa \\ stephan.wagner@math.uu.se, swagner@sun.ac.za}
\title{Refined enumeration of $k$-plane trees and $k$-noncrossing trees}
\date{}

\begin{document}
\maketitle 
\begin{abstract}
A $k$-\emph{plane tree} is a plane tree whose vertices are assigned labels between $1$ and $k$ in such a way that the sum of the labels along any edge is no greater than $k+1$. These trees are known to be related to $(k+1)$-ary trees, and they are counted by a generalised version of the Catalan numbers. We prove a surprisingly simple refined counting formula, where we count trees with a prescribed number of labels of each kind. Several corollaries are derived from this formula, and an analogous theorem is proven for $k$-\emph{noncrossing trees}, a similarly defined family of labelled noncrossing trees that are related to $(2k+1)$-ary trees.
\end{abstract}

\section{Introduction and Preliminaries}
\label{sec:Intro}

Many combinatorial objects that are counted by the Catalan numbers have $k$-ary analogues. Heubach, Li and Mansour list several such examples in \cite{Heubach2008Staircase}, among them $k$-ary trees, different families of lattice paths, nonintersecting arc sequences, and certain types of Young diagrams.

\medskip

The family of $k$-plane trees, which was first considered in \cite{GuProdingerWagner2010Bijections}, is another example that leads to $k$-ary analogues of the Catalan numbers. It is the family of all labelled plane trees (rooted trees where the order of branches matters) with vertex labels in the set $[k] = \{1,2,\ldots,k\}$ and the restriction that the sum of the labels along any edge is never greater than $k+1$. Figure~\ref{fig:4plane} shows an example of a $4$-plane tree. 

\medskip

Note that $1$-plane trees are simply plane trees where every vertex is labelled $1$, which are counted by the Catalan numbers. Moreover, we note that a plane tree with vertex labels in $\{1,2\}$ is a $2$-plane tree if and only if the vertices labelled $2$ form an independent set. Therefore, the total number of $2$-plane trees is the same as the total number of independent sets in all plane trees, which was determined in~\cite{Klazar1997Twelve}.

\medskip

The number of $k$-plane trees with $n$ vertices is the generalised Catalan number
$$\frac{1}{n-1} \binom{(k+1)(n-1)}{n} = \frac{k}{n} \binom{(k+1)(n-1)}{n-1},$$
and there is a similar formula for the number of $k$-plane trees with $n$ vertices whose root is labelled $h$:
$$\frac{k+1-h}{kn-h+1} \binom{(k+1)n-h-1}{n-1}.$$
In particular, we obtain the number of $(k+1)$-ary trees (trees where every internal vertex has precisely $k+1$ children) with $n-1$ internal vertices when $h = k$. An explicit bijection is provided in~\cite{GuProdingerWagner2010Bijections}.

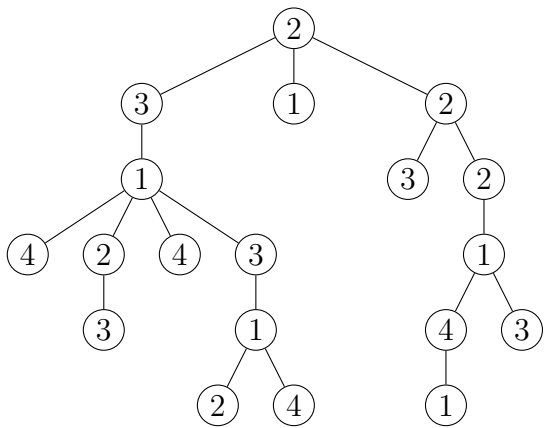
\begin{figure}[htbp]
\centering
\begin{tikzpicture}
\node[draw,circle,inner sep=2pt] (u1) at (0,0) {2};
\node[draw,circle,inner sep=2pt] (u2) at (-2,-1) {3};
\node[draw,circle,inner sep=2pt] (u3) at (0,-1) {1};
\node[draw,circle,inner sep=2pt] (u4) at (2,-1) {2};
\node[draw,circle,inner sep=2pt] (u5) at (-2,-2) {1};
\node[draw,circle,inner sep=2pt] (u6) at (1.5,-2) {3};
\node[draw,circle,inner sep=2pt] (u7) at (2.5,-2) {2};
\node[draw,circle,inner sep=2pt] (u8) at (-3.5,-3) {4};
\node[draw,circle,inner sep=2pt] (u9) at (-2.5,-3) {2};
\node[draw,circle,inner sep=2pt] (u10) at (-1.5,-3) {4};
\node[draw,circle,inner sep=2pt] (u11) at (-0.5,-3) {3};
\node[draw,circle,inner sep=2pt] (u12) at (2.5,-3) {1};
\node[draw,circle,inner sep=2pt] (u13) at (-2.5,-4) {3};
\node[draw,circle,inner sep=2pt] (u14) at (-0.5,-4) {1};
\node[draw,circle,inner sep=2pt] (u15) at (2,-4) {4};
\node[draw,circle,inner sep=2pt] (u16) at (3,-4) {3};
\node[draw,circle,inner sep=2pt] (u17) at (-1,-5) {2};
\node[draw,circle,inner sep=2pt] (u18) at (0,-5) {4};
\node[draw,circle,inner sep=2pt] (u19) at (2,-5) {1};

\draw (u1)--(u2);
\draw (u1)--(u3);
\draw (u1)--(u4);
\draw (u2)--(u5);
\draw (u4)--(u6);
\draw (u4)--(u7);
\draw (u5)--(u8);
\draw (u5)--(u9);
\draw (u5)--(u10);
\draw (u5)--(u11);
\draw (u7)--(u12);
\draw (u9)--(u13);
\draw (u11)--(u14);
\draw (u12)--(u15);
\draw (u12)--(u16);
\draw (u14)--(u17);
\draw (u14)--(u18);
\draw (u15)--(u19);
\end{tikzpicture}
\caption{An example of a $4$-plane tree.}\label{fig:4plane}
\end{figure}

\medskip

The aim of this paper is to provide refined counting formulas for $k$-plane trees based on the number of occurrences of each label. Perhaps surprisingly, there is an explicit product formula for the number of $k$-plane trees with prescribed multiplicities of all labels. The main theorem reads as follows:

\begin{thm}\label{thm:main1}
Let $n > 1$. The number of $k$-plane trees with root label $h$, $\ell_i$ vertices labelled $i$ ($i \in [k]$) and $n = \ell_1+\ell_2 + \cdots + \ell_k$ vertices in total is given by
\begin{multline*}
\frac{\ell_h}{n(2n-1)} \prod_{r=1}^{\lceil k/2 \rceil} \binom{2n-1-\sum_{j=1}^{r-1} \ell_j - \sum_{j=k+2-r}^k \ell_j}{\ell_r} \\
\prod_{r=1}^{h-1} \binom{\sum_{j=1}^{r} \ell_j + \sum_{j=k+1-r}^k \ell_j - 1}{\ell_{k+1-r}} \prod_{r=h}^{\lfloor k/2 \rfloor} \binom{\sum_{j=1}^{r} \ell_j + \sum_{j=k+1-r}^k \ell_j}{\ell_{k+1-r}}
\end{multline*}
if $h \leq \lceil k/2 \rceil$, and by
\begin{multline*}
\frac{\ell_h}{(n-1)(2n-1)} \prod_{r=1}^{k+1-h} \binom{2n-1-\sum_{j=1}^{r-1} \ell_j - \sum_{j=k+2-r}^k \ell_j}{\ell_r} \\
\prod_{r=k+2-h}^{\lceil k/2 \rceil} \binom{2n-2-\sum_{j=1}^{r-1} \ell_j - \sum_{j=k+2-r}^k \ell_j}{\ell_r}
\prod_{r=1}^{\lfloor k/2 \rfloor} \binom{\sum_{j=1}^{r} \ell_j + \sum_{j=k+1-r}^k \ell_j - 1}{\ell_{k+1-r}}
\end{multline*}
otherwise.
The total number of $k$-plane trees with $\ell_i$ vertices labelled $i$ ($i \in [k]$) and $n = \ell_1+\ell_2 + \cdots + \ell_k$ vertices in total is 
$$\frac{1}{n-1} 
\prod_{r=1}^{\lceil k/2 \rceil} \binom{2n-2-\sum_{j=1}^{r-1} \ell_j - \sum_{j=k+2-r}^k \ell_j}{\ell_r}
\prod_{r=1}^{\lfloor k/2 \rfloor} \binom{\sum_{j=1}^{r} \ell_j + \sum_{j=k+1-r}^k \ell_j - 1}{\ell_{k+1-r}}.$$
\end{thm}

Let us remark here that empty products are always considered to be $1$, and empty sums are considered to be $0$. To illustrate the result in a special case, let us give the formula for the total number of $4$-plane trees with $\ell_1,\ell_2,\ell_3,\ell_4$ vertices labelled $1,2,3,4$ respectively and $n = \ell_1 + \ell_2 + \ell_3 + \ell_4$ vertices in total:
\begin{multline*}
\frac{1}{n-1} \binom{2n-2}{\ell_1} \binom{2n-2-\ell_1-\ell_4}{\ell_2} \binom{\ell_1 + \ell_2 + \ell_3 + \ell_4 - 1}{\ell_3} \binom{\ell_1 + \ell_4 - 1}{\ell_4} \\
= \frac{1}{n-1} \binom{2n-2}{\ell_1} \binom{2n-2-\ell_1-\ell_4}{\ell_2} \binom{n-1}{\ell_3} \binom{\ell_1 + \ell_4 - 1}{\ell_4}.
\end{multline*}

Let us also mention again that $2$-plane trees are precisely $\{1,2\}$-labelled plane trees where the vertices labelled $2$ form an independent set. Therefore, for $k=2$, we obtain formulas due to Kirschenhofer, Prodinger and Tichy \cite{Kirschenhofer1986Fibonacci} for the total number of independent sets of a given size (containing the root, not containing the root, and total) in plane trees with $n$ vertices as special cases.

\medskip

Theorem~\ref{thm:main1} will be proven in Section~\ref{sec:plane} by first establishing a system of functional equations, which can be solved explicitly by means of a suitable substitution. The formula finally follows by an application of the Lagrange-B\"urmann formula. We will derive a number of corollaries from the formula in Theorem~\ref{thm:main1}, in particular on the average number of occurrences of a specific label. 

\medskip

The family of $k$-plane trees can also be bijectively related to lattice paths with up-steps of the form $(1,1)$ and down-steps of the form $(1,-k)$, see \cite{GuProdingerWagner2010Bijections}. In \cite{Heuberger2022Enumeration}, such paths are enumerated by the $y$-coordinates of the down-steps modulo $k$. Interestingly, the number of paths with exactly $a_i$ down-steps at level $i$ modulo $k$ for every $i$ turns out to be given by a similar, albeit somewhat different, product formula as those in Theorem~\ref{thm:main1}.

\medskip

In Section~\ref{sec:noncrossing}, we consider a similar family of trees called $k$-noncrossing trees: recall that a noncrossing tree is a tree whose vertices $v_1,v_2,\ldots,v_n$ can be arranged as points on a circle (in this order) with the edges represented by line segments between these points that do not intersect at interior points. In analogy to $k$-plane trees, one defines $k$-noncrossing trees as noncrossing trees whose vertices are labelled with labels in $[k]$ in such a way that the labels of two adjacent vertices $v_i, v_j$ with $i < j$ cannot add up to a sum greater than $k+1$ if the path from the root $v_1$ to $v_j$ contains $v_i$ (this includes the case that $i = 1$). Figure~\ref{fig:3noncr} shows an example of a $3$-noncrossing tree. Note that it contains two edges between vertices labelled $2$ and $3$ respectively that would not be allowed in a $k$-plane tree, but are allowed here because the path from the root moves from the vertex with higher index to the vertex with lower index.

\medskip

The special case $k = 2$ was considered in \cite{YanLiu2009Noncrossing}, where a bijection between $2$-noncrossing trees with a root labelled $2$ and $5$-ary trees was constructed. The more general case was studied in \cite{Pang2010kNoncrossing}, see also \cite{Okoth2022Enumeration}.

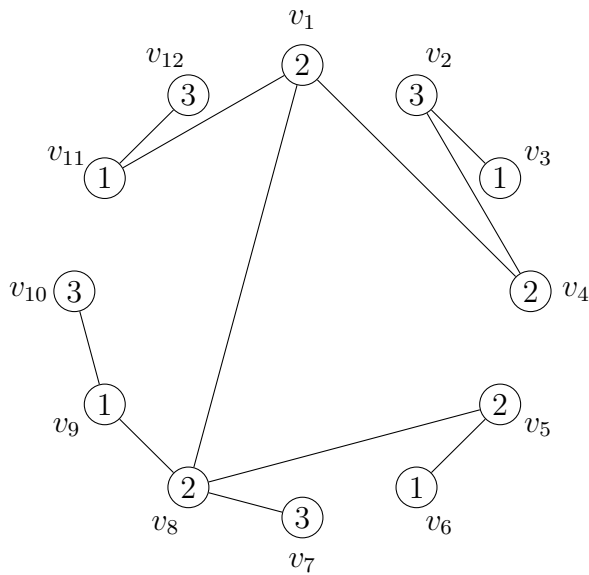
\begin{figure}[htbp]
\centering
\begin{tikzpicture}
\node[draw,circle,inner sep=2pt] (u1) at (0,3) {2};
\node[draw,circle,inner sep=2pt] (u2) at (1.5,2.6) {3};
\node[draw,circle,inner sep=2pt] (u3) at (2.6,1.5) {1};
\node[draw,circle,inner sep=2pt] (u4) at (3,0) {2};
\node[draw,circle,inner sep=2pt] (u5) at (2.6,-1.5) {2};
\node[draw,circle,inner sep=2pt] (u6) at (1.5,-2.6) {1};
\node[draw,circle,inner sep=2pt] (u7) at (0,-3) {3};
\node[draw,circle,inner sep=2pt] (u8) at (-1.5,-2.6) {2};
\node[draw,circle,inner sep=2pt] (u9) at (-2.6,-1.5) {1};
\node[draw,circle,inner sep=2pt] (u10) at (-3,0) {3};
\node[draw,circle,inner sep=2pt] (u11) at (-2.6,1.5) {1};
\node[draw,circle,inner sep=2pt] (u12) at (-1.5,2.6) {3};

\draw (u1)--(u4);
\draw (u1)--(u8);
\draw (u1)--(u11);
\draw (u2)--(u3);
\draw (u2)--(u4);
\draw (u5)--(u6);
\draw (u5)--(u8);
\draw (u7)--(u8);
\draw (u8)--(u9);
\draw (u9)--(u10);
\draw (u11)--(u12);

\node at (0,3.6) {$v_1$};
\node at (1.8,3.1) {$v_2$};
\node at (3.1,1.8) {$v_3$};
\node at (3.6,0) {$v_4$};
\node at (3.1,-1.8) {$v_5$};
\node at (1.8,-3.1) {$v_6$};
\node at (0,-3.6) {$v_7$};
\node at (-1.8,-3.1) {$v_8$};
\node at (-3.1,-1.8) {$v_9$};
\node at (-3.6,0) {$v_{10}$};
\node at (-3.1,1.8) {$v_{11}$};
\node at (-1.8,3.1) {$v_{12}$};
\end{tikzpicture}
\caption{An example of a $3$-noncrossing tree.}\label{fig:3noncr}
\end{figure}

\medskip

In analogy to Theorem~\ref{thm:main1}, we will also be counting $k$-noncrossing trees by the number of vertices of each label. The resulting formulas are quite similar and again surprisingly explicit.

\begin{thm}\label{thm:main2}
Let $n > 1$. The number of $k$-noncrossing trees with root label $h$, $\ell_i$ vertices labelled $i$ ($i \in [k]$) and $n = \ell_1+\ell_2 + \cdots + \ell_k$ vertices in total is given by
\begin{align*}
&\frac{2\ell_h}{(2n-1)(4n-3)} \prod_{r=1}^{\lceil k/2 \rceil} \binom{3n-2-\sum_{j=1}^{r-1} \ell_j - \sum_{j=k+2-r}^k \ell_j}{\ell_r} \\
&\qquad \prod_{r=1}^{h-1} \binom{n - 2 + \sum_{j=1}^{r} \ell_j + \sum_{j=k+1-r}^k \ell_j}{\ell_{k+1-r}} \prod_{r=h}^{\lfloor k/2 \rfloor} \binom{n - 1 + \sum_{j=1}^{r} \ell_j + \sum_{j=k+1-r}^k \ell_j}{\ell_{k+1-r}} \\
&\quad - \frac{\ell_h }{(4n-3)(3n-2-\sum_{j=1}^{h-1} \ell_j - \sum_{j=k+2-h}^k \ell_j)} \prod_{r=1}^{\lfloor k/2 \rfloor} \binom{n - 1 + \sum_{j=1}^{r} \ell_j + \sum_{j=k+1-r}^k \ell_j}{\ell_{k+1-r}} \\
&\qquad \prod_{r=1}^{h-1} \binom{3n-3-\sum_{j=1}^{r-1} \ell_j - \sum_{j=k+2-r}^k \ell_j}{\ell_r} 
\prod_{r=h}^{\lceil k/2 \rceil} \binom{3n-2-\sum_{j=1}^{r-1} \ell_j - \sum_{j=k+2-r}^k \ell_j}{\ell_r}
\end{align*}
if $h \leq \lceil k/2 \rceil$, and by
\begin{align*}
&\frac{\ell_h}{(n-1)(4n-3)} \prod_{r=1}^{\lfloor k/2 \rfloor} \binom{n - 2 + \sum_{j=1}^{r} \ell_j + \sum_{j=k+1-r}^k \ell_j}{\ell_{k+1-r}} \\
&\qquad \prod_{r=1}^{k+1-h} \binom{3n-2-\sum_{j=1}^{r-1} \ell_j - \sum_{j=k+2-r}^k \ell_j}{\ell_r} 
\prod_{r=k+2-h}^{\lceil k/2 \rceil} \binom{3n-3-\sum_{j=1}^{r-1} \ell_j - \sum_{j=k+2-r}^k \ell_j}{\ell_r} \\
&\quad - \frac{\ell_h }{(4n-3)(n - 1 + \sum_{j=1}^{k+1-h} \ell_j + \sum_{j=h}^k \ell_j)} 
\prod_{r=1}^{\lceil k/2 \rceil} \binom{3n-3-\sum_{j=1}^{r-1} \ell_j - \sum_{j=k+2-r}^k \ell_j}{\ell_r} \\
&\qquad \prod_{r=1}^{k+1-h} \binom{n - 1 + \sum_{j=1}^{r} \ell_j + \sum_{j=k+1-r}^k \ell_j}{\ell_{k+1-r}} \prod_{r=k+2-h}^{\lfloor k/2 \rfloor} \binom{n - 2 + \sum_{j=1}^{r} \ell_j + \sum_{j=k+1-r}^k \ell_j}{\ell_{k+1-r}}
\end{align*}
otherwise.
The total number of $k$-noncrossing trees with $\ell_i$ vertices labelled $i$ ($i \in [k]$) and $n = \ell_1+\ell_2 + \cdots + \ell_k$ vertices in total is 
\begin{multline*}
\frac{1}{n-1} \prod_{r=1}^{\lceil k/2 \rceil} \binom{3n-3-\sum_{j=1}^{r-1} \ell_j - \sum_{j=k+2-r}^k \ell_j}{\ell_r}
\prod_{r=1}^{\lfloor k/2 \rfloor} \binom{n - 2 + \sum_{j=1}^{r} \ell_j + \sum_{j=k+1-r}^k \ell_j}{\ell_{k+1-r}} \\
- \frac{1}{2n-1} \prod_{r=1}^{\lceil k/2 \rceil} \binom{3n-2-\sum_{j=1}^{r-1} \ell_j - \sum_{j=k+2-r}^k \ell_j}{\ell_r}
\prod_{r=1}^{\lfloor k/2 \rfloor} \binom{n - 1 + \sum_{j=1}^{r} \ell_j + \sum_{j=k+1-r}^k \ell_j}{\ell_{k+1-r}}.
\end{multline*}
\end{thm}

This theorem will be proven in Section~\ref{sec:noncrossing}, using similar techniques as in the proof of Theorem~\ref{thm:main1}, and several corollaries will follow as well. Let us again illustrate the formula in the special case $k=4$, where the formula for the total number of $4$-noncrossing trees with $\ell_1,\ell_2,\ell_3,\ell_4$ vertices labeled $1,2,3,4$ respectively and $n$ vertices in total is
\begin{multline*}
\frac{1}{n-1} \binom{3n-3}{\ell_1} \binom{3n-3-\ell_1 -\ell_4}{\ell_2}
\binom{2n - 2}{\ell_3} \binom{n - 2 + \ell_1 + \ell_4}{\ell_4} \\
- \frac{1}{2n-1} \binom{3n-2}{\ell_1} \binom{3n-2-\ell_1 - \ell_4}{\ell_2}
\binom{2n - 1}{\ell_3} \binom{n - 1 + \ell_1 + \ell_4}{\ell_4}.
\end{multline*}

\section{Plane trees}
\label{sec:plane}

The key to proving Theorem~\ref{thm:main1} is a system of equations for the multivariate generating functions of $k$-plane trees with a given root label. We fix $k$ and let $\mathcal{P}_r$ denote the set of $k$-plane trees whose root is labelled $r$. Moreover, we let $h_i(T)$ be the number of vertices labelled $i$ in a tree $T$, and $|T|$ the total number of vertices of $T$. Define
$$P_r = P_r(z,x_1,x_2,\ldots,x_k) = \sum_{T \in \mathcal{P}_r} z^{|T|} \prod_{i=1}^k x_i^{h_i(T)}.$$
Now note that a tree $T \in \mathcal{P}_r$ can be decomposed into the root (labelled $r$) and a (possibly empty) sequence of branches that are again $k$-plane trees, with root labels in $[k+1-r] = \{1,2,\ldots,k+1-r\}$. Thus we have
\begin{equation}\label{eq:mainequation}
P_r = x_r z \sum_{j \geq 0} (P_1 + P_2 + \cdots + P_{k+1-r})^j = \frac{x_r z}{1-P_1-P_2 - \cdots - P_{k+1-r}}
\end{equation}
for all $r \in [k]$. 

\medskip

Now let us set
$$F_{k,i}(t) = 1 + \Big( \sum_{j=i}^{\lceil k/2\rceil} x_j - \sum_{j=i}^{\lfloor k/2\rfloor} x_{k+1-j} \Big) t$$
and
$$G_{k,i}(t) = 1 + \Big( \sum_{j=i+1}^{\lceil k/2\rceil} x_j - \sum_{j=i}^{\lfloor k/2\rfloor} x_{k+1-j} \Big) t$$
for $1 \leq i \leq \lceil k/2\rceil$. These expressions satisfy the recursions
\begin{equation}\label{eq:rec1}
F_{k,i+1}(t) = G_{k,i}(t) + x_{k+1-i}t \qquad \text{and} \qquad G_{k,i}(t) = F_{k,i}(t) - x_i t,
\end{equation}
as well as
\begin{equation}\label{eq:rec2}
F_{k,i+1}(t) = F_{k,i}(t) + (x_{k+1-i} - x_i)t \qquad \text{and} \qquad G_{k,i+1}(t) =  G_{k,i}(t) + (x_{k+1-i} - x_{i+1})t.
\end{equation}
We can use these recursions to continue the definition of $F_{k,i}(t)$ and $G_{k,i}(t)$ to greater values of $i$: generally, we set
$$F_{k,i}(t) = F_{k,1}(t) + \sum_{j=1}^{i-1} (x_{k+1-j} - x_j)t$$
and
$$G_{k,i}(t) = G_{k,1}(t) + \sum_{j=1}^{i-1} (x_{k+1-j} - x_{j+1})t.$$
Both~\eqref{eq:rec1} and~\eqref{eq:rec2} remain satisfied. Since
$$\sum_{j=i}^{k+1-i} (x_{k+1-j} - x_j) = 0 \qquad \text{and} \qquad \sum_{j=i}^{k-i} (x_{k+1-j} - x_{j+1}) = 0,$$
we see that the following symmetry properties hold:
\begin{equation}\label{eq:sym}
F_{k,i}(t) = F_{k,k+2-i}(t)\qquad \text{and} \qquad G_{k,i}(t) = G_{k,k+1-i}(t).
\end{equation}
Moreover, it is important to observe that
\begin{equation}\label{eq:midpoint}
F_{k,k/2+1} = 1\quad (k \text{ even})\qquad \text{and} \qquad  G_{k,(k+1)/2} = 1\quad (k \text{ odd}).
\end{equation}

The key to the proof of Theorem~\ref{thm:main1} is the substitution
\begin{equation}\label{eq:subst}
P_1 = \frac{x_1 A}{F_{k,1}(A)}
\end{equation}
for a suitable power series $A$. One can easily solve the equation for $A$ to show that this power series actually exists (and that it is unique).

\medskip

As it turns out, we can express $P_1,P_2,\ldots,P_{k-1}$ in terms of $A$ as well and also set up a functional equation for $A$ that is amenable to an application of the Lagrange inversion formula.

\begin{pro}\label{pro:formulas_for_Pi}
The power series $P_1,P_2,\ldots,P_k$ can be expressed in terms of $A$ and the variables $x_1,x_2,\ldots,x_k$ and  $z$ in the following way: for $1 \leq h \leq k$,
\begin{align*}
P_h &= x_h A \prod_{i=1}^{h} F_{k,i}(A)^{-1} \prod_{i=1}^{h-1} G_{k,i}(A), \\
P_{k+1-h} &= x_{k+1-h} z \prod_{i=1}^{h} F_{k,i}(A) \prod_{i=1}^h G_{k,i}(A)^{-1}.
\end{align*}
\end{pro}

\begin{proof}
We use induction on $h$. For $h = 1$, the first equation is exactly our substitution~\eqref{eq:subst}, while the second equation follows from~\eqref{eq:mainequation} for $r = k$ and an application of~\eqref{eq:rec1}:
$$P_k = \frac{x_k z}{1-P_1} = \frac{x_k z}{1 - \frac{x_1 A}{F_{k,1}(A)}} = \frac{x_k z F_{k,1}(A)}{F_{k,1}(A) - x_1 A} = \frac{x_k z F_{k,1}(A)}{G_{k,1}(A)}.$$
For the induction step, use~\eqref{eq:mainequation} with $r = h$ and $r = h+1$ respectively, which yields
\begin{align*}
1 - P_1 - P_2 - \cdots - P_{k+1-h} = \frac{x_h z}{P_h}, \\
1 - P_1 - P_2 - \cdots - P_{k-h} = \frac{x_{h+1} z}{P_{h+1}}.
\end{align*}
Now take the difference:
\begin{equation}\label{eq:after_elimination}
P_{k+1-h} = \frac{x_{h+1}z}{P_{h+1}} - \frac{x_hz}{P_h}.
\end{equation}
After some simple manipulations, this gives us
\begin{equation}\label{eq:induction_step}
P_{h+1} = \frac{x_{h+1}z}{P_{k+1-h} + \frac{x_h z}{P_h}}.
\end{equation}
Now it only remains to apply the induction hypothesis and simplify:
\begin{align*}
P_{h+1} &= \frac{x_{h+1}z}{x_{k+1-h} z \prod_{i=1}^h F_{k,i}(A) \prod_{i=1}^h G_{k,i}(A)^{-1} + \frac{z}{A} \prod_{i=1}^h F_{k,i}(A) \prod_{i=1}^{h-1} G_{k,i}(A)^{-1}} \\
&= \frac{x_{h+1}A}{x_{k+1-h}A + G_{k,h}(A)} \prod_{i=1}^h F_{k,i}(A)^{-1} \prod_{i=1}^h G_{k,i}(A) \\
&= \frac{x_{h+1}A}{F_{k,h+1}(A)} \prod_{i=1}^h F_{k,i}(A)^{-1} \prod_{i=1}^h G_{k,i}(A) \\
&= x_{h+1}A \prod_{i=1}^{h+1} F_{k,i}(A)^{-1} \prod_{i=1}^h G_{k,i}(A).
\end{align*}
Likewise, replacing $h$ by $k-h$ in~\eqref{eq:after_elimination} gives us
$$P_{h+1} = \frac{x_{k+1-h}z}{P_{k+1-h}} - \frac{x_{k-h}z}{P_{k-h}}.$$
Thus
$$P_{k-h} = \frac{x_{k-h}z}{\frac{x_{k+1-h}z}{P_{k+1-h}} - P_{h+1}}.$$
Now plug in~\eqref{eq:induction_step} and apply the induction hypothesis. Again, we obtain the desired formula after some further manipulations.
\end{proof}

\begin{cor}\label{cor:identity_for_A}
The power series $A$ satisfies the equation
$$A = z \prod_{i=1}^{\lceil k/2 \rceil} F_{k,i}(A)^2 \prod_{i=1}^{\lfloor k/2 \rfloor} G_{k,i}(A)^{-2}.$$
\end{cor}

\begin{proof}
Replace $h$ by $k+1-h$ in the second equation of Proposition~\ref{pro:formulas_for_Pi} to obtain another representation for $P_h$:
$$P_h = x_h z \prod_{i=1}^{k+1-h} F_{k,i}(A) \prod_{i=1}^{k+1-h} G_{k,i}(A)^{-1}.$$
Equating the two expressions for $P_h$ yields
$$A = z \prod_{i=1}^h F_{k,i}(A) \prod_{i=1}^{k+1-h} F_{k,i}(A) \prod_{i=1}^{h-1} G_{k,i}(A)^{-1} \prod_{i=1}^{k+1-h} G_{k,i}(A)^{-1}$$
Now we apply the symmetry properties~\eqref{eq:sym} to the second and fourth product:
\begin{align*}
A &= z \prod_{i=1}^h F_{k,i}(A) \prod_{i=h+1}^{k+1} F_{k,i}(A) \prod_{i=1}^{h-1} G_{k,i}(A)^{-1} \prod_{i=h}^{k} G_{k,i}(A)^{-1} \\
&= z \prod_{i=1}^{k+1} F_{k,i}(A) \prod_{i=1}^{k} G_{k,i}(A)^{-1}.
\end{align*}
Applying the symmetry properties once again, and noting that $F_{k,k/2+1}$ can be left out if $k$ is even, while $G_{k,(k+1)/2}$ can be left out if $k$ is odd (by~\eqref{eq:midpoint}), we end up with
$$A = z \prod_{i=1}^{\lceil k/2 \rceil} F_{k,i}(A)^2 \prod_{i=1}^{\lfloor k/2 \rfloor} G_{k,i}(A)^{-2},$$
completing the proof.
Note that $h$ could have been chosen arbitrarily for this purpose.
\end{proof}

We can now proceed with the proof of our first main theorem.

\begin{proof}[Proof of Theorem~\ref{thm:main1}]
We are now ready to apply the Lagrange-B\"urmann formula \cite[Corollary 5.4.3]{Stanley2}, based on Proposition~\ref{pro:formulas_for_Pi} and Corollary~\ref{cor:identity_for_A}. Let us first recall this formula: if $A$ satisfies an implicit equation of the form $A = z \Phi(A)$, then
\begin{equation}\label{eq:lag-bur}
[z^n] f(A) = \frac{1}{n} [t^{n-1}] f'(t) \Phi(t)^n.
\end{equation}
Now suppose first that $h \leq \lceil k/2 \rceil$. In view of Proposition~\ref{pro:formulas_for_Pi} and Corollary~\ref{cor:identity_for_A}, we can apply~\eqref{eq:lag-bur} with
$$\Phi(t) = \prod_{i=1}^{\lceil k/2 \rceil} F_{k,i}(t)^2 \prod_{i=1}^{\lfloor k/2 \rfloor} G_{k,i}(t)^{-2}$$
and
$$f(t) = x_h t \prod_{i=1}^{h} F_{k,i}(t)^{-1} \prod_{i=1}^{h-1} G_{k,i}(t)$$
in order to compute the coefficients of $P_h$. The derivative $f'$ is determined by means of logarithmic differentiation. It is also important to observe that
$$\frac{F'_{k,i}(t)}{F_{k,i}(t)} = \frac{1}{t} \Big( 1 - \frac{1}{F_{k,i}(t)} \Big)\qquad \text{and} \qquad \frac{G'_{k,i}(t)}{G_{k,i}(t)} = \frac{1}{t} \Big( 1 - \frac{1}{G_{k,i}(t)} \Big),$$
which yields
\begin{align*}
f'(t) &= f(t) \Big( \frac1{t} - \sum_{i=1}^h \frac{F'_{k,i}(t)}{F_{k,i}(t)} + \sum_{i=1}^{h-1} \frac{G'_{k,i}(t)}{G_{k,i}(t)} \Big) \\
&= f(t) \Big( \frac1{t} - \frac1{t} \sum_{i=1}^h \Big( 1 - \frac{1}{F_{k,i}(t)} \Big) + \frac1{t} \sum_{i=1}^{h-1} \Big( 1 - \frac{1}{G_{k,i}(t)} \Big) \Big) \\
&= \frac{f(t)}{t} \Big( \sum_{i=1}^h \frac{1}{F_{k,i}(t)} - \sum_{i=1}^{h-1} \frac{1}{G_{k,i}(t)} \Big).
\end{align*}
So we have
\begin{align*}
[z^n x_1^{\ell_1} \cdots x_k^{\ell_k}] P_h &= \frac{1}{n} [t^{n-1} x_1^{\ell_1} \cdots x_k^{\ell_k}] f'(t) \Phi(t)^n \\
&= \frac{1}{n} [t^{n-1} x_1^{\ell_1} \cdots x_k^{\ell_k}] 
x_h \prod_{i=1}^{h} F_{k,i}(t)^{-1} \prod_{i=1}^{h-1} G_{k,i}(t) \Big( \sum_{i=1}^h \frac{1}{F_{k,i}(t)} - \sum_{i=1}^{h-1} \frac{1}{G_{k,i}(t)} \Big) \\
&\qquad \prod_{i=1}^{\lceil k/2 \rceil} F_{k,i}(t)^{2n} \prod_{i=1}^{\lfloor k/2 \rfloor} G_{k,i}(t)^{-2n} \\
&= \frac{1}{n} [t^{n-1} x_1^{\ell_1} \cdots x_h^{\ell_h-1} \cdots x_k^{\ell_k}] \Big( \prod_{i=h+1}^{\lceil k/2 \rceil} F_{k,i}(t) \prod_{i=h}^{\lfloor k/2 \rfloor} G_{k,i}(t)^{-1} \Big)^{2n} \\
&\qquad \Big( \prod_{i=1}^{h} F_{k,i}(t) \prod_{i=1}^{h-1} G_{k,i}(t)^{-1} \Big)^{2n-1} \Big( \sum_{i=1}^h \frac{1}{F_{k,i}(t)} - \sum_{i=1}^{h-1} \frac{1}{G_{k,i}(t)} \Big).
\end{align*}
At this point, we can drop the variable $t$ (equivalently, set $t=1$), since the coefficient of $t^{n-1} x_1^{\ell_1} \cdots x_h^{\ell_h-1} \cdots x_k^{\ell_k}$ is only nonzero if $\ell_1 + \cdots + \ell_k = n$. Thus we will also only write $F_{k,i}$ and $G_{k,i}$ instead of $F_{k,i}(1)$ and $G_{k,i}(1)$ respectively.

\medskip

Next we note that $F_{k,1},F_{k,2},\ldots,F_{k,h}$ and $G_{k,1},G_{k,2},\ldots,G_{k,h-1}$ are the only factors that contain the variable $x_h$. Moreover, using logarithmic differentiation once again, one finds that
$$
\frac{\partial}{\partial x_h} \Big( \prod_{i=1}^{h} F_{k,i} \prod_{i=1}^{h-1} G_{k,i}^{-1} \Big)^{2n-1}
= (2n-1) \Big( \prod_{i=1}^{h} F_{k,i} \prod_{i=1}^{h-1} G_{k,i}^{-1} \Big)^{2n-1}\Big( \sum_{i=1}^h \frac{1}{F_{k,i}} - \sum_{i=1}^{h-1} \frac{1}{G_{k,i}} \Big).
$$
For any power series $X$, the coefficient of $x_h^{\ell_h-1}$ in $\frac{\partial}{\partial x_h} X$ is precisely $\ell_h$ times the coefficient of $x_h^{\ell_h}$ in $X$. Thus we get
$$
[z^n x_1^{\ell_1} \cdots x_k^{\ell_k}] P_h
= \frac{\ell_h}{n(2n-1)} [x_1^{\ell_1} \cdots x_k^{\ell_k}]
\Big( \prod_{i=1}^{h} F_{k,i} \prod_{i=1}^{h-1} G_{k,i}^{-1} \Big)^{2n-1} \Big( \prod_{i=h+1}^{\lceil k/2 \rceil} F_{k,i} \prod_{i=h}^{\lfloor k/2 \rfloor} G_{k,i}^{-1} \Big)^{2n}.
$$
Now we finally start extracting coefficients. First of all, we observe that $x_1$ only occurs in the factor $F_{k,1}$. Moreover, we have
$$F_{k,1}^{2n-1} = (x_1 + G_{k,1})^{2n-1},$$
so the coefficient of $x_1^{\ell_1}$ is $\binom{2n-1}{\ell_1} G_{k,1}^{2n-1-\ell_1}$, giving us
\begin{multline*}
[z^n x_1^{\ell_1} \cdots x_k^{\ell_k}] P_h = \frac{\ell_h}{n(2n-1)} \binom{2n-1}{\ell_1} [x_2^{\ell_2} \cdots x_k^{\ell_k}]
G_{k,1}^{-\ell_1} \\
\Big( \prod_{i=2}^{h} F_{k,i} \prod_{i=2}^{h-1} G_{k,i}^{-1} \Big)^{2n-1} \Big( \prod_{i=h+1}^{\lceil k/2 \rceil} F_{k,i} \prod_{i=h}^{\lfloor k/2 \rfloor} G_{k,i}^{-1} \Big)^{2n}.
\end{multline*}
Among the remaining factors, $G_{k,1}$ is the only one that contains $x_k$, and we have
\begin{align*}
[x_k^{\ell_k}] G_{k,1}^{-\ell_1} &= [x_k^{\ell_k}] (F_{k,2} - x_k)^{-\ell_1} = [x_k^{\ell_k}] F_{k,2}^{-\ell_1} \Big( 1 - \frac{x_k}{F_{k,2}} \Big)^{-\ell_1} \\
&= F_{k,2}^{-\ell_1} \binom{-\ell_1}{\ell_k} \Big( - \frac{1}{F_{k,2}}\Big)^{\ell_k} = \binom{\ell_1+\ell_k-1}{\ell_k} F_{k,2}^{-\ell_1-\ell_k}.
\end{align*}
It follows that
\begin{multline*}
[z^n x_1^{\ell_1} \cdots x_k^{\ell_k}] P_h = \frac{\ell_h}{n(2n-1)} \binom{2n-1}{\ell_1} \binom{\ell_1+\ell_k-1}{\ell_k} [x_2^{\ell_2} \cdots x_{k-1}^{\ell_{k-1}}]
F_{k,2}^{2n-\ell_1-\ell_k-1} \\
\Big( \prod_{i=3}^{h} F_{k,i} \prod_{i=2}^{h-1} G_{k,i}^{-1} \Big)^{2n-1} \Big( \prod_{i=h+1}^{\lceil k/2 \rceil} F_{k,i} \prod_{i=h}^{\lfloor k/2 \rfloor} G_{k,i}^{-1} \Big)^{2n}.
\end{multline*}
We can now continue in this way, considering the variables $x_1,x_k,x_2,x_{k-1},x_3,\ldots$ in this order. At the end, we have precisely the first formula in Theorem~\ref{thm:main1}.

\medskip

The case that $h > \lceil k/2 \rceil$ is treated in a similar way: since
$$P_{h} = x_{h} z \prod_{i=1}^{k+1-h} F_{k,i}(A) \prod_{i=1}^{k+1-h} G_{k,i}(A)^{-1}$$
in this case by the second equation of Proposition~\ref{pro:formulas_for_Pi}, we apply the Lagrange-B\"urmann formula with
$$f(t) = x_h \prod_{i=1}^{k+1-h} F_{k,i}(t) \prod_{i=1}^{k+1-h} G_{k,i}(t)^{-1}$$
and the same function $\Phi$ as before. This yields
\begin{align*}
[z^n x_1^{\ell_1} \cdots x_k^{\ell_k}] P_h &= [z^{n-1} x_1^{\ell_1} \cdots x_k^{\ell_k}] f(A) \\
&= \frac{1}{(n-1)} [t^{n-2} x_1^{\ell_1} \cdots x_k^{\ell_k}] f'(t) \Phi(t)^n.
\end{align*}
The remaining steps are completely analogous to the first case.

\medskip

Finally, we consider the total number of $k$-plane trees, without taking the root label into account. Here, we first observe that the generating function is
$$P_1+P_2 + \cdots + P_k = 1 - \frac{x_1z}{P_1}$$
in view of~\eqref{eq:mainequation} (for $r = 1$). In terms of $A$, this becomes
\begin{equation}\label{eq:total_gf}
P_1+P_2 + \cdots + P_k = 1 - \frac{z F_{k,1}(A)}{A}
\end{equation}
by \eqref{eq:subst}. Thus
$$[z^n x_1^{\ell_1} \cdots x_k^{\ell_k}] (P_1+P_2+\cdots+P_k) = - [z^{n-1} x_1^{\ell_1} \cdots x_k^{\ell_k}] \frac{F_{k,1}(A)}{A}.$$
Now we apply the Lagrange-B\"urmann formula once again. Noting that
$$\frac{\partial}{\partial t} \frac{F_{k,1}(t)}{t} = - \frac{1}{t^2},$$
we obtain
\begin{align*}
[z^n x_1^{\ell_1} \cdots x_k^{\ell_k}] (P_1+P_2+\cdots+P_k) &= \frac{1}{n-1} [t^{n-2} x_1^{\ell_1} \cdots x_k^{\ell_k}] t^{-2} \Phi(t)^{n-1} \\
&= \frac{1}{n-1} [t^{n} x_1^{\ell_1} \cdots x_k^{\ell_k}] \Phi(t)^{n-1},
\end{align*}
with the same function $\Phi$ as before. Again, the remaining steps are completely analogous.
\end{proof}

We conclude the section with corollaries of Theorem~\ref{thm:main1} that follow by specialisation. First of all, the following formulas from \cite{GuProdingerWagner2010Bijections} follow easily  by ignoring the variables $x_1,x_2,\ldots,x_k$.

\begin{cor}\label{cor:total}
For every positive integer $n$, the total number of $k$-plane trees with $n$ vertices is
$$\frac{k}{n} \binom{(k+1)(n-1)}{n-1}.$$
The number of $k$-plane trees with $n$ vertices whose root is labelled $h$ is
$$\frac{k+1-h}{kn-h+1} \binom{(k+1)n-h-1}{n-1}.$$
\end{cor}

\begin{proof}
Since the number of labels of each kind is no longer relevant, we can set $x_1 = x_2 = \cdots = x_k = 1$. We get
$$F_{k,i}(t) = \begin{cases} 1 & k \text{ even,} \\ 1 + t & k \text{ odd,} \end{cases}$$
as well as
$$G_{k,i}(t) = \begin{cases} 1 - t & k \text{ even,} \\ 1 & k \text{ odd,} \end{cases}$$
for all values of $i$. Consider the case that $k$ is even, the other being similar. 
Corollary~\ref{cor:identity_for_A} gives us
$$A = z (1-A)^{-k}.$$
So by the Lagrange-B\"urmann formula, we have
\begin{align*}
[z^n] (P_1+P_2 + \cdots + P_k) &= [z^n] \Big( 1 - \frac{z F_{k,1}(A)}{A} \Big) = - [z^{n-1}] A^{-1} \\
&= \frac{1}{n-1} [t^{n-2}] t^{-2} (1-t)^{-k(n-1)} = \frac{1}{n-1} [t^{n}] (1-t)^{-k(n-1)} \\
&= \frac{1}{n-1} \binom{(k+1)(n-1)}{n} = \frac{k}{n} \binom{(k+1)(n-1)}{n-1}.
\end{align*}
This is exactly the first formula. Considering the coefficients of $P_h$, which count trees whose root is labelled $h$, we have
$$P_h = A (1-A)^{h-1}$$
by Proposition~\ref{pro:formulas_for_Pi}. Thus
\begin{align*}
[z^n] P_h & = \frac{1}{n} [t^{n-1}] \big((1-t)^{h-1} - t(h-1)(1-t)^{h-2}\big) (1-t)^{-k n}. \\
&= \frac{1}{n} [t^{n-1}] (1-ht) (1-t)^{-k n + h - 2} \\
&= \frac{1}{n} \Big( \binom{(k+1) n - h}{n-1} - h \binom{(k+1) n - h - 1}{n-2} \Big)  \\
&= \frac{1}{n} \binom{(k+1) n - h - 1}{n-1} \Big( \frac{(k+1) n - h}{k n - h + 1} - \frac{h(n-1)}{k n - h + 1} \Big) \\
&= \frac{k - h + 1}{k n - h + 1} \binom{(k+1) n - h - 1}{n-1}.
\end{align*}

\end{proof}

Next, we count $k$-plane trees by occurrences of a single label.

\begin{cor}\label{cor:single_label}
For every $n > 1$, the total number of $k$-plane trees with $n$ vertices of which $\ell$ are labelled $h$ is equal to
$$\frac{1}{n-1} \sum_{r=0}^{n-\ell} \binom{2(h-1)(n-1)+r-1}{r} \binom{2(n-1)-r}{\ell} \binom{(k+1-2h)(n-1)}{n-r-\ell}$$
if $h \leq \lceil k/2 \rceil$, and equal to
$$\frac{1}{n-1} \sum_{r=0}^{n-\ell} \binom{2(k+1-h)(n-1)}{r} \binom{r+\ell-1}{\ell} \binom{(2h-k-1)(n-1)-r-\ell}{n-r-\ell}$$
otherwise.
\end{cor}

\begin{proof}
Let us give the proof in the case that $k$ is odd and $h \leq \lceil k/2 \rceil = \frac{k+1}{2}$, the other cases being similar. Since we are only interested in vertices labelled $h$, we set $x_i = 1$ for all $i$ except $h$. Then we get
$$F_{k,i}(t) = \begin{cases} 1+x_ht & i \leq h, \\ 1+t & i > h,\end{cases} \qquad \text{and} \qquad G_{k,i}(t)  = \begin{cases} 1 + (x_h-1)t & i < h, \\ 1 & i \geq h. \end{cases}$$
Now we have to determine
$$[z^n x_h^{\ell}] (P_1 + P_2 + \cdots + P_k).$$
Using~\eqref{eq:total_gf} and the Lagrange-B\"urmann formula once again, we find that this equals
\begin{multline*}
[z^n x_h^{\ell}] (P_1 + P_2 + \cdots + P_k) \\
= \frac{1}{n-1} [t^n x_h^{\ell}] (1+x_h t)^{2h(n-1)} (1+t)^{(k+1-2h)(n-1)} (1+(x_h-1)t)^{-2(h-1)(n-1)}.
\end{multline*}
Now we extract the coefficient as follows:
\begin{align*}
[z^n x_h^{\ell}] &(P_1 + P_2 + \cdots + P_k) \\
&= \frac{1}{n-1} [t^n x_h^{\ell}] \Big( 1 - \frac{t}{1+x_h t}\Big)^{-2(h-1)(n-1)} 
(1+x_h t)^{2(n-1)} (1+t)^{(k+1-2h)(n-1)} \\
&= \frac{1}{n-1} [t^n x_h^{\ell}]  \sum_{r \geq 0} \binom{2(h-1)(n-1)+r-1}{r} t^r (1+x_ht)^{2(n-1)-r} (1+t)^{(k+1-2h)(n-1)} \\
&= \frac{1}{n-1} \sum_{r \geq 0} \binom{2(h-1)(n-1)+r-1}{r} [t^{n-r} x_h^{\ell}] (1+x_ht)^{2(n-1)-r} (1+t)^{(k+1-2h)(n-1)} \\
&= \frac{1}{n-1} \sum_{r \geq 0} \binom{2(h-1)(n-1)+r-1}{r} \binom{2(n-1)-r}{\ell} [t^{n-r-\ell}] (1+t)^{(k+1-2h)(n-1)} \\
&= \frac{1}{n-1} \sum_{r=0}^{n-\ell} \binom{2(h-1)(n-1)+r-1}{r} \binom{2(n-1)-r}{\ell} \binom{(k+1-2h)(n-1)}{n-r-\ell}.
\end{align*}
\end{proof}

\begin{cor}\label{cor:average}
For every $n > 1$, the average number of vertices labelled $h$ in $k$-plane trees with $n$ vertices is
$$\frac{2(k+1-h)n}{k(k+1)}.$$
\end{cor}

\begin{proof}
As in the previous proof, we only consider the case that $k$ is odd and $h \leq \lceil k/2 \rceil = \frac{k+1}{2}$. Instead of extracting coefficients, we take the derivative with respect to $x_h$ and plug in $x_h=1$ in order to determine the total number of vertices labelled $h$ in all $k$-plane trees. All other variables $x_i$ are immediately taken to be $1$. This gives us
\begin{align*}
[z^n] &\frac{\partial}{\partial x_h} (P_1+\cdots + P_k) \Big|_{x_1=\cdots=x_k = 1} \\
&= \frac{1}{n-1} [t^n] \frac{\partial}{\partial x_h} (1+x_h t)^{2h(n-1)} (1+t)^{(k+1-2h)(n-1)} (1+(x_h-1)t)^{-2(h-1)(n-1)} \Big|_{x_h=1} \\
&= [t^n] 2t \big(1-(h-1)t \big) (1+t)^{(k+1)(n-1)-1} \\
&= 2[t^{n-1}] \big(1-(h-1)t \big) (1+t)^{(k+1)(n-1)-1} \\
&= 2 \binom{(k+1)(n-1)-1}{n-1} - 2(h-1) \binom{(k+1)(n-1)-1}{n-2}.
\end{align*}

Dividing by the total number of $k$-plane trees (as given in Corollary~\ref{cor:total}), we obtain the stated formula.
\end{proof}

It is also possible to derive formulas for the variance of the number of vertices labelled $h$, as well as covariances of two different label counts. However, the formulas are somewhat unwieldy. Moreover, one could also take the root label into account in Corollary~\ref{cor:single_label} and Corollary~\ref{cor:average}. Instead of stating the most general result (which would be rather lengthy), we illustrate this in the special case $k=3$.

\begin{cor}\label{variance_plane}
Let $n > 1$. Variances and covariances of the number of vertices labelled $1,2,3$ respectively in $3$-plane trees with $n$ vertices are given in the following table:
\begin{center}
\begin{tabular}{c|ccc}
& $1$ & $2$ & $3$ \\
\hline
$1$ & $\frac{n(3n-4)}{4(4n-5)}$ & $-\frac{n}{6}$ & $- \frac{n(n-2)}{12(4n-5)}$ \\
$2$ & $-\frac{n}{6}$ & $\frac{2n(4n-3)}{9(3n-2)}$ & $- \frac{n(7n-6)}{18(3n-2)}$ \\
$3$ & $- \frac{n(n-2)}{12(4n-5)}$ & $- \frac{n(7n-6)}{18(3n-2)}$ & $\frac{n(5n-4)(13n-18)}{36(3n-2)(4n-5)}$ \\
\end{tabular}
\end{center}
\end{cor}

\begin{proof}
We recall from the proof of Theorem~\ref{thm:main1} that
$$[z^nx_1^{\ell_1}x_2^{\ell_2}x_3^{\ell_3}] (P_1+P_2+P_3) = \frac{1}{n-1} [t^nx_1^{\ell_1}x_2^{\ell_2}x_3^{\ell_3}] \Phi(t)^{n-1},$$
where
$$\Phi(t) = (1+(x_1+x_2-x_3)t)^2(1+x_2t)^2(1+(x_2-x_3)t)^{-2}$$
in the special case $k=3$. For the variance of the number of vertices labelled $h$, we need to compute the second moment, which is
\begin{equation}\label{eq:second_moment}
\frac{[z^n] \frac{\partial^2}{\partial x_h^2} (P_1+P_2+P_3) |_{x_1=x_2=x_3 = 1} + [z^n] \frac{\partial}{\partial x_h} (P_1+P_2+P_3) |_{x_1=x_2=x_3 = 1}}{[z^n] (P_1+P_2+P_3) |_{x_1=x_2=x_3 = 1}},
\end{equation}
and then subtract the square of the mean. Likewise, the mixed moment of the number of vertices labelled $h$ and the number of vertices labelled $i$ is
$$\frac{[z^n] \frac{\partial^2}{\partial x_h \partial x_i} (P_1+P_2+P_3) |_{x_1=x_2=x_3 = 1}}{[z^n] (P_1+P_2+P_3) |_{x_1=x_2=x_3 = 1}},$$
from which we subtract the product of the means to obtain the covariance.

\medskip

Let us only show the calculations for the variance of the number of vertices labelled $1$ explicitly. Here, we obtain
\begin{align*}
[z^n] \frac{\partial^2}{\partial x_1^2} (P_1+P_2+P_3) |_{x_1=x_2=x_3 = 1} &= \frac{1}{n-1} [t^n] 2(n-1)(2n-3)t^2 (1+t)^{4n-6} \\
&= 2(2n-3) [t^{n-2}](1+t)^{4n-6} = 2(2n-3) \binom{4n-6}{n-2}.
\end{align*}
We already found earlier that
$$[z^n] \frac{\partial}{\partial x_1} (P_1+P_2+P_3) |_{x_1=x_2=x_3 = 1} = 2 \binom{4n-5}{n-1}$$
and
$$[z^n](P_1+P_2+P_3) |_{x_1=x_2=x_3 = 1} = \frac{1}{n-1} \binom{4n-4}{n}.$$
Plugging everything into~\eqref{eq:second_moment} and simplifying, we find a formula for the second moment and thus in turn for the variance.
\end{proof}

\begin{cor}
Let $n > 1$. The average number of vertices labelled $1,2,3$ respectively in $3$-plane trees with $n$ vertices whose root is labelled $1,2,3$ respectively is given in the following table:
\begin{center}
\begin{tabular}{c|ccc}
& $1$ & $2$ & $3$ \\
\hline
root label $1$ & $\frac{n^2}{2 n-1}$ & $\frac{n-1}{3}$ & $\frac{(n-1) (n+1)}{3 (2 n-1)}$ \\
root label $2$ & $\frac{n-1}{2}$ & $\frac{n^2+3 n-1}{3 n}$ & $\frac{(n-2) (n-1)}{6 n}$ \\
root label $3$ & $\frac{n}{2}$ & $\frac{(n-2) (n-1)}{3 n-1}$ & $\frac{n^2+5 n-4}{2 (3 n-1)}$
\end{tabular}
\end{center}
\end{cor}

\begin{proof}
We recall from the proof of Theorem~\ref{thm:main1} that
$$P_1 = \frac{x_1 A}{1 + (x_1+x_2-x_3)A} ,\ P_2 = \frac{x_2 A(1 + (x_2-x_3)A)}{(1 + x_2 A) (1 + (x_1+x_2-x_3)A)},$$
$$P_3 = \frac{x_3 A (1 + (x_2-x_3) A )}{(1 + x_2A)^2 (1 + (x_1+x_2-x_3)A)},$$
with
$$A = z (1+(x_1+x_2-x_3)A)^2(1+x_2A)^2(1+(x_2-x_3)A)^{-2}.$$
In order to determine the desired mean values, we need the coefficients of the partial derivatives $\frac{\partial}{\partial x_h} P_i |_{x_1=x_2=x_3 = 1}$. 
We will show the details of the calculations in one of the cases again: the number of vertices labelled $1$ in $3$-plane trees whose root label is $1$.
Since
$$\frac{\partial}{ \partial t} \frac{x_1 t}{1 + (x_1+x_2-x_3)t} = \frac{x_1}{(1+(x_1+x_2-x_3)t)^2},$$
we have
$$[z^n] P_1 = \frac{1}{n} [t^{n-1}] \frac{x_1}{(1+(x_1+x_2-x_3)t)^2} (1+(x_1+x_2-x_3)t)^{2n}(1+x_2t)^{2n}(1+(x_2-x_3)t)^{-2n}.$$
Now differentiate with respect to $x_1$ and set $x_1 = x_2 = x_3 = 1$:
\begin{align*}
\frac{\partial}{\partial x_1} [z^n] P_1 \Big|_{x_1 = x_2 = x_3 = 1} &= \frac{1}{n} [t^{n-1}] (1+(2n-1)t) (1+t)^{4n-3} \\
&= \frac{1}{n} \Big( \binom{4n-3}{n-1} + (2n-1) \binom{4n-3}{n-2} \Big).
\end{align*}
Dividing by the total number of $3$-plane trees with $n$ vertices and root label $1$, which is $\frac{1}{n} \binom{4n-2}{n-1}$, we obtain the mean number of vertices labelled $1$, namely
$$\frac{\frac{1}{n}( \binom{4n-3}{n-1} + (2n-1) \binom{4n-3}{n-2})}{\frac{1}{n} \binom{4n-2}{n-1}} = \frac{n^2}{2n-1}.$$

 \end{proof}

\section{Noncrossing trees}
\label{sec:noncrossing}

Our aim in this section is to obtain analogous results for noncrossing trees. In particular, we will prove Theorem~\ref{thm:main2}. To this end, we set up a system of functional equations once again. We fix $k$ and let $\mathcal{N}_r$ denote the set of $k$-noncrossing trees whose root is labelled $r$. As before, we let $h_i(T)$ be the number of vertices labelled $i$ in a tree $T$, and $|T|$ the total number of vertices of $T$. Finally, we define the following generating functions in analogy to the generating functions $P_r$ in the previous section.
$$N_r = N_r(z,x_1,x_2,\ldots,x_k) = \sum_{T \in \mathcal{N}_r} z^{|T|} \prod_{i=1}^k x_i^{h_i(T)}.$$
The decomposition of $k$-noncrossing trees is slightly more subtle than that of plane trees. Every noncrossing tree can be decomposed into the root and a sequence of so-called \emph{butterflies}, which are pairs of noncrossing trees joined at a common root. The roots of these butterflies are the children $v_{i_1},v_{i_2},\ldots,v_{i_r}$ of the root $v_1$. One part of the butterfly rooted at $v_{i_j}$ contains all those vertices whose indices are less than or equal to than $i_j$ (i.e., vertices $v_s$ with $s \leq i_j$), the other contains all those vertices whose indices are greater than or equal to $i_j$ (i.e., vertices $v_s$ with $s \geq i_j$). We refer to them as \emph{lower} and \emph{upper part} of a butterfly; they only have the vertex $v_{i_j}$ in common. See Figure~\ref{fig:butter} for an illustration.

\begin{figure}[htbp]
\centering
\begin{tikzpicture}
\node[draw,circle,inner sep=2pt] (u1) at (0,3) {\phantom{2}};
\node[draw,circle,inner sep=2pt] (u2) at (3,0) {\phantom{2}};
\node[draw,circle,inner sep=2pt] (u3) at (0,-3) {\phantom{2}};
\node[draw,circle,inner sep=2pt] (u4) at (-3,0) {\phantom{2}};

\draw (u1)--(u2);
\draw (u1)--(u3);
\draw (u1)--(u4);

\node at (0,3.6) {$v_1$};
\node at (3.6,0) {$v_{i_1}$};
\node at (0,-3.6) {$v_{i_2}$};
\node at (-3.6,0) {$v_{i_3}$};

\draw (u2)--(2.5,2)--(3.5,2)--(u2);
\draw (u2)--(2.5,-2)--(3.5,-2)--(u2);
\draw (u3)--(-2,-2.5)--(-2,-3.5)--(u3);
\draw (u3)--(2,-2.5)--(2,-3.5)--(u3);
\draw (u4)--(-2.5,2)--(-3.5,2)--(u4);
\draw (u4)--(-2.5,-2)--(-3.5,-2)--(u4);

\node at (4.5,1) {lower part};
\node at (4.5,-1) {upper part};
\end{tikzpicture}
\caption{The butterfly decomposition.}\label{fig:butter}
\end{figure}
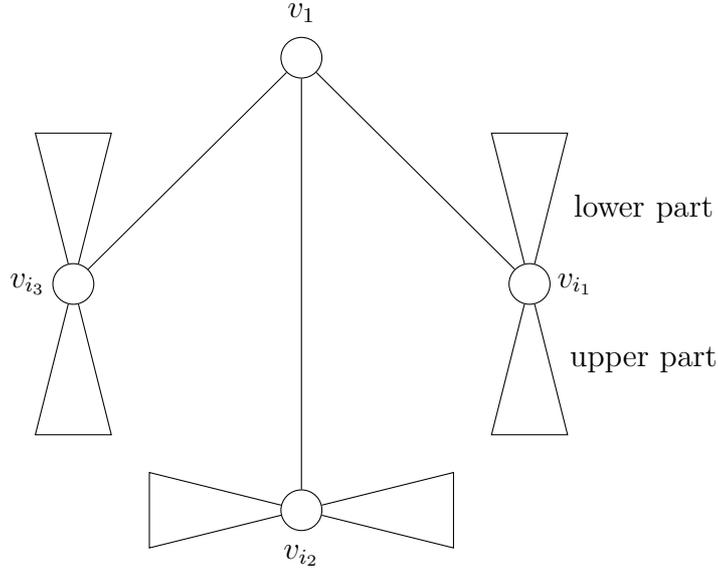

\medskip

Each of the two parts of a butterfly can be seen as a noncrossing tree. However, because of the definition of $k$-noncrossing trees, which involves the order of the vertices on the circle, the two parts are slightly different. The upper part containing vertices $v_s$ with $s \geq i_j$ forms a proper $k$-noncrossing tree. However, the lower part is only almost a $k$-noncrossing tree: the rule on labels not adding up to values greater than $k+1$ does not apply to the root edges (for all other edges, the rule is exactly as it is in a proper $k$-noncrossing tree). Thus the lower part is not necessarily a proper $k$-noncrossing tree, but it always becomes one by changing the root label to $1$ (if it is not already $1$). This is because a label $1$ can always be paired with any other label along an edge. Thus we find that a butterfly with root label $r$ has generating function
$$N_r \cdot \frac{N_1}{x_1 z}.$$
The first factor represents the upper part of the butterfly, the second factor the lower part, but excluding the root. This is achieved by the denominator.

\medskip

Arguing as in the previous section, we see that a $k$-noncrossing tree with root label $r$ has branches that are butterflies with root labels in $[k+1-r]$. Thus 
\begin{equation}\label{eq:mainequation_nc}
N_r = x_r z \sum_{j \geq 0} \Big( \frac{N_1}{x_1 z} (N_1 + N_2 + \cdots + N_{k+1-r}) \Big)^j = \frac{x_r z}{1-\frac{N_1}{x_1 z}(N_1 + N_2 + \cdots + N_{k+1-r})}
\end{equation}
for all $r \in [k]$. 

\medskip

We use the same expressions $F_{k,i}(t)$ and $G_{k,i}(t)$ as in the previous section. The substitution, however, will be slightly different. In analogy to~\eqref{eq:subst}, we set
\begin{equation}\label{eq:subst_nc}
N_1 = x_1 \sqrt{\frac{z B}{F_{k,1}(B)}}.
\end{equation}
Again, it is not hard to verify that there exists a suitable power series $B$ that satisfies this equation, and that it is unique. Next, we prove an analogue of Proposition~\ref{pro:formulas_for_Pi}.

\begin{pro}\label{pro:formulas_for_Ni}
The power series $N_1,N_2,\ldots,N_k$ can be expressed in terms of $B$ and the variables $x_1,x_2,\ldots,x_k$ and  $z$ in the following way: for $1 \leq h \leq k$,
\begin{align*}
N_h &= x_h \sqrt{\frac{z B}{F_{k,1}(B)}} \prod_{i=2}^{h} F_{k,i}(B)^{-1} \prod_{i=1}^{h-1} G_{k,i}(B), \\
N_{k+1-h} &= x_{k+1-h} z \prod_{i=1}^{h} F_{k,i}(B) \prod_{i=1}^h G_{k,i}(B)^{-1}.
\end{align*}
\end{pro}

\begin{proof}
The proof is analogous to that of Proposition~\ref{pro:formulas_for_Pi} by induction on $h$. For $h = 1$, the first equation is exactly our substitution~\eqref{eq:subst_nc}, while the second equation follows from~\eqref{eq:mainequation_nc} for $r = k$ and an application of~\eqref{eq:rec2}:
$$N_k = \frac{x_k z}{1-\frac{N_1^2}{x_1 z}} = \frac{x_k z}{1 - \frac{x_1^2 z B}{x_1 z F_{k,1}(B)}} = \frac{x_k z F_{k,1}(B)}{F_{k,1}(B) - x_1 B} = \frac{x_k z F_{k,1}(B)}{G_{k,1}(B)}.$$
For the induction step, use~\eqref{eq:mainequation_nc} with $r = h$ and $r = h+1$ respectively, which yields
\begin{align*}
1 - \frac{N_1}{x_1 z}(N_1 + N_2 + \cdots + N_{k+1-h}) = \frac{x_h z}{N_h}, \\
1 - \frac{N_1}{x_1 z}(N_1 + N_2 + \cdots + N_{k-h}) = \frac{x_{h+1} z}{N_{h+1}}.
\end{align*}
Now take the difference:
\begin{equation}\label{eq:after_elimination_nc}
\frac{N_1 N_{k+1-h}}{x_1 z} = \frac{x_{h+1}z}{N_{h+1}} - \frac{x_hz}{N_h}.
\end{equation}
After some manipulations, this gives us
\begin{equation}\label{eq:induction_step_nc}
N_{h+1} = \frac{x_{h+1}z}{\frac{N_1 N_{k+1-h}}{x_1 z} + \frac{x_h z}{N_h}}.
\end{equation}
Now it only remains to apply the induction hypothesis and simplify:
\begin{align*}
N_{h+1} &= \frac{x_{h+1}z}{\sqrt{\frac{z B}{F_{k,1}(B)}} x_{k+1-h} \prod_{i=1}^h F_{k,i}(B) \prod_{i=1}^h G_{k,i}(B)^{-1} + \sqrt{\frac{zF_{k,1}(B)}{B}} \prod_{i=2}^h F_{k,i}(B) \prod_{i=1}^{h-1} G_{k,i}(B)^{-1}} \\
&= \frac{x_{h+1}}{x_{k+1-h}B + G_{k,h}(B)} \sqrt{\frac{z B}{F_{k,1}(B)}} \prod_{i=2}^h F_{k,i}(B)^{-1} \prod_{i=1}^h G_{k,i}(B) \\
&= \frac{x_{h+1}}{F_{k,h+1}(B)} \sqrt{\frac{z B}{F_{k,1}(B)}} \prod_{i=2}^h F_{k,i}(B)^{-1} \prod_{i=1}^h G_{k,i}(B) \\
&= x_{h+1} \sqrt{\frac{z B}{F_{k,1}(B)}} \prod_{i=2}^{h+1} F_{k,i}(B)^{-1} \prod_{i=1}^h G_{k,i}(B).
\end{align*}
Likewise, replacing $h$ by $k-h$ in~\eqref{eq:after_elimination_nc} gives us
$$\frac{N_1 N_{h+1}}{x_1 z} = \frac{x_{k+1-h}z}{N_{k+1-h}} - \frac{x_{k-h}z}{N_{k-h}}.$$
Thus
$$N_{k-h} = \frac{x_{k-h}z}{\frac{x_{k+1-h}z}{N_{k+1-h}} - \frac{N_1 N_{h+1}}{x_1 z}}.$$
Now plug in~\eqref{eq:induction_step_nc} and apply the induction hypothesis. Again, we obtain the desired formula after some further manipulations.
\end{proof}

\begin{cor}\label{cor:identity_for_B_nc}
The power series $B$ satisfies the equation
$$B = z F_{k,1}(B)^3 \prod_{i=2}^{\lceil k/2 \rceil} F_{k,i}(B)^4 \prod_{i=1}^{\lfloor k/2 \rfloor} G_{k,i}(B)^{-4}.$$
\end{cor}

\begin{proof}
In analogy to Corollary~\ref{cor:identity_for_A}, we use the two representations for $N_h$ provided by Proposition~\ref{pro:formulas_for_Ni}:
$$N_h = x_h \sqrt{\frac{z B}{F_{k,1}(B)}} \prod_{i=2}^{h} F_{k,i}(B)^{-1} \prod_{i=1}^{h-1} G_{k,i}(B) = x_{h} z \prod_{i=1}^{k+1-h} F_{k,i}(B) \prod_{i=1}^{k+1-h} G_{k,i}(B)^{-1}.$$
Applying the symmetry relations~\eqref{eq:sym}, we get
$$\sqrt{\frac{B}{F_{k,1}(B)}} \prod_{i=2}^{h} F_{k,i}(B)^{-1} \prod_{i=1}^{h-1} G_{k,i}(B) = \sqrt{z} \prod_{i=h+1}^{k+1} F_{k,i}(B) \prod_{i=h}^{k} G_{k,i}(B)^{-1}.$$
Squaring and simplifying yields
$$B = z F_{k,1}(B)\prod_{i=2}^{k+1} F_{k,i}(B)^2 \prod_{i=1}^{k} G_{k,i}(B)^{-2}.$$
Applying the symmetry properties as well as~\eqref{eq:midpoint} once again, we arrive at the stated formula:
$$B = z F_{k,1}(B)^3 \prod_{i=2}^{\lceil k/2 \rceil} F_{k,i}(B)^4 \prod_{i=1}^{\lfloor k/2 \rfloor} G_{k,i}(B)^{-4}.$$
As in the proof of Corollary~\ref{cor:identity_for_A}, $h$ was arbitrary in these calculations.
\end{proof}

We are now ready to prove the second main theorem of this paper. It is very similar to the proof of Theorem~\ref{thm:main1}, with some small modifications.

\begin{proof}[Proof of Theorem~\ref{thm:main2}]
As before, we apply the Lagrange-B\"urmann formula, based on Proposition~\ref{pro:formulas_for_Ni} and Corollary~\ref{cor:identity_for_B_nc}. We start with the case that $h \leq \lceil k/2 \rceil$. In view of Proposition~\ref{pro:formulas_for_Ni} and Corollary~\ref{cor:identity_for_B_nc}, we can apply~\eqref{eq:lag-bur} with
$$\Phi(t) = F_{k,1}(t)^3 \prod_{i=2}^{\lceil k/2 \rceil} F_{k,i}(t)^4 \prod_{i=1}^{\lfloor k/2 \rfloor} G_{k,i}(t)^{-4}$$
and
$$f(t) = x_h \sqrt{\frac{t}{F_{k,1}(t)}} \prod_{i=2}^{h} F_{k,i}(t)^{-1} \prod_{i=1}^{h-1} G_{k,i}(t),$$
but we have to take the coefficient of $z^{n-1/2}$ in view of the factor $\sqrt{z}$ in the formula for $N_h$. Once again, we apply logarithmic differentiation to determine the derivative of $f(t)$. We find that
\begin{align*}
f'(t) &= f(t) \Big( \frac1{2t} - \frac{F'_{k,1}(t)}{2F_{k,1}(t)} - \sum_{i=2}^h \frac{F'_{k,i}(t)}{F_{k,i}(t)} + \sum_{i=1}^{h-1} \frac{G'_{k,i}(t)}{G_{k,i}(t)} \Big) \\
&= f(t) \Big( \frac1{2t} - \frac{1}{2t} \Big( 1 - \frac{1}{F_{k,1}(t)} \Big) - \frac1{t} \sum_{i=2}^h \Big( 1 - \frac{1}{F_{k,i}(t)} \Big) + \frac1{t} \sum_{i=1}^{h-1} \Big( 1 - \frac{1}{G_{k,i}(t)} \Big) \Big) \\
&= \frac{f(t)}{t} \Big( \frac{1}{2F_{k,1}(t)} + \sum_{i=2}^h \frac{1}{F_{k,i}(t)} - \sum_{i=1}^{h-1} \frac{1}{G_{k,i}(t)} \Big).
\end{align*}
So we have
\begin{align}
[z^n x_1^{\ell_1} \cdots x_k^{\ell_k}] N_h &= \frac{1}{n-\frac12} [t^{n-3/2} x_1^{\ell_1} \cdots x_k^{\ell_k}] f'(t) \Phi(t)^{n-1/2} \nonumber \\
&= \frac{2}{2n-1} [t^{n-3/2} x_1^{\ell_1} \cdots x_k^{\ell_k}] 
x_h \sqrt{\frac{1}{tF_{k,1}(t)}} \prod_{i=2}^{h} F_{k,i}(t)^{-1} \prod_{i=1}^{h-1} G_{k,i}(t)  \nonumber \\
&\qquad \Big( \frac{1}{2F_{k,1}(t)} + \sum_{i=2}^h \frac{1}{F_{k,i}(t)} - \sum_{i=1}^{h-1} \frac{1}{G_{k,i}(t)} \Big) \nonumber \\
&\qquad F_{k,1}(t)^{3n-3/2} \prod_{i=2}^{\lceil k/2 \rceil} F_{k,i}(t)^{4n-2} \prod_{i=1}^{\lfloor k/2 \rfloor} G_{k,i}(t)^{2-4n} \nonumber \\
&= \frac{2}{2n-1} [t^{n-1} x_1^{\ell_1} \cdots x_h^{\ell_h-1} \cdots x_k^{\ell_k}] \Big( \prod_{i=h+1}^{\lceil k/2 \rceil} F_{k,i}(t) \prod_{i=h}^{\lfloor k/2 \rfloor} G_{k,i}(t)^{-1} \Big)^{4n-2}\nonumber \\
&\qquad F_{k,1}(t)^{3n-2} \Big( \prod_{i=2}^{h} F_{k,i}(t) \prod_{i=1}^{h-1} G_{k,i}(t)^{-1} \Big)^{4n-3} \label{eq:after_lagrange_nc} \\
&\qquad \Big(  \frac{1}{2F_{k,1}(t)} + \sum_{i=2}^h \frac{1}{F_{k,i}(t)} - \sum_{i=1}^{h-1} \frac{1}{G_{k,i}(t)} \Big). \nonumber
\end{align}
We remark that we are applying the Lagrange-B\"urmann formula, somewhat unusually, in a situation where we have half-integer exponents in our power series, but it is not difficult to verify that it works equally well. At this point, we drop the variable $t$ again (by setting $t=1$), since we know that the coefficient of $t^{n-1} x_1^{\ell_1} \cdots x_h^{\ell_h-1} \cdots x_k^{\ell_k}$ is only nonzero when $\ell_1 + \cdots + \ell_k = n$. We will also write $F_{k,i}$ and $G_{k,i}$ instead of $F_{k,i}(1)$ and $G_{k,i}(1)$ again.

\medskip

As in the proof of Theorem~\ref{thm:main1}, we observe that $F_{k,1},F_{k,2},\ldots,F_{k,h}$ and $G_{k,1},G_{k,2},\ldots,G_{k,h-1}$ are the only factors that contain the variable $x_h$. Moreover, using logarithmic differentiation once again, one finds that
\begin{multline*}
\frac{\partial}{\partial x_h} F_{k,1}^{3n-2} \Big( \prod_{i=2}^{h} F_{k,i} \prod_{i=1}^{h-1} G_{k,i}^{-1} \Big)^{4n-3} \\
= (4n-3) F_{k,1}^{3n-2} \Big( \prod_{i=2}^{h} F_{k,i} \prod_{i=1}^{h-1} G_{k,i}^{-1} \Big)^{4n-3}
\Big( \frac{3n-2}{(4n-3)F_{k,1}} + \sum_{i=2}^h \frac{1}{F_{k,i}} - \sum_{i=1}^{h-1} \frac{1}{G_{k,i}} \Big).
\end{multline*}
Now we split the expression in~\eqref{eq:after_lagrange_nc} into two parts, one of which can be seen as a derivative with respect to $x_h$ in the same way as in the proof of Theorem~\ref{thm:main1}:
\begin{align*}
[z^n x_1^{\ell_1} \cdots x_k^{\ell_k}] N_h &= \frac{2}{2n-1} [x_1^{\ell_1} \cdots x_h^{\ell_h-1} \cdots x_k^{\ell_k}] \Big( \prod_{i=h+1}^{\lceil k/2 \rceil} F_{k,i} \prod_{i=h}^{\lfloor k/2 \rfloor} G_{k,i}^{-1} \Big)^{4n-2}  \\
&\qquad F_{k,1}^{3n-2} \Big( \prod_{i=2}^{h} F_{k,i} \prod_{i=1}^{h-1} G_{k,i}^{-1} \Big)^{4n-3} \Big( \frac{3n-2}{(4n-3)F_{k,1}} + \sum_{i=2}^h \frac{1}{F_{k,i}} - \sum_{i=1}^{h-1} \frac{1}{G_{k,i}} \Big) \\
&\quad -\frac{2}{2n-1} [x_1^{\ell_1} \cdots x_h^{\ell_h-1} \cdots x_k^{\ell_k}] \Big( \prod_{i=h+1}^{\lceil k/2 \rceil} F_{k,i} \prod_{i=h}^{\lfloor k/2 \rfloor} G_{k,i}^{-1} \Big)^{4n-2}  \\
&\qquad F_{k,1}^{3n-2} \Big( \prod_{i=2}^{h} F_{k,i} \prod_{i=1}^{h-1} G_{k,i}^{-1} \Big)^{4n-3} \cdot \frac{2n-1}{2(4n-3)F_{k,1}} \\
&= \frac{2\ell_h}{(2n-1)(4n-3)} [x_1^{\ell_1} \cdots x_k^{\ell_k}] \Big( \prod_{i=h+1}^{\lceil k/2 \rceil} F_{k,i} \prod_{i=h}^{\lfloor k/2 \rfloor} G_{k,i}^{-1} \Big)^{4n-2} \\
&\qquad F_{k,1}^{3n-2} \Big( \prod_{i=2}^{h} F_{k,i} \prod_{i=1}^{h-1} G_{k,i}^{-1} \Big)^{4n-3} \\
&\quad - \frac{1}{4n-3} [x_1^{\ell_1} \cdots x_h^{\ell_h-1} \cdots x_k^{\ell_k}]  \Big( \prod_{i=h+1}^{\lceil k/2 \rceil} F_{k,i} \prod_{i=h}^{\lfloor k/2 \rfloor} G_{k,i}^{-1} \Big)^{4n-2} \\
&\qquad F_{k,1}^{3n-3} \Big( \prod_{i=2}^{h} F_{k,i} \prod_{i=1}^{h-1} G_{k,i}^{-1} \Big)^{4n-3}.
\end{align*}

Now we can extract coefficients from both products in the same way as in the proof of Theorem~\ref{thm:main2}, i.e. by considering the variables in the order $x_1,x_k,x_2,x_{k-1},\ldots$. 

\medskip

The derivation of the formula in the case that $h > \lceil k/2 \rceil$ is similar: we start from the representation
$$N_h = x_h z \prod_{i=1}^{k+1-h} F_{k,i}(B) \prod_{i=1}^{k+1-h} G_{k,i}(B)^{-1},$$
which means that we can apply the Lagrange-B\"urmann formula with
$$f(t) = x_h \prod_{i=1}^{k+1-h} F_{k,i}(t) \prod_{i=1}^{k+1-h} G_{k,i}(t)^{-1}$$
and the same function $\Phi$ as before. In view of the factor $z$ in the expression for $N_h$, we have to extract the coefficient of $z^{n-1}$. Thus
\begin{align*}
[z^n x_1^{\ell_1} \cdots x_k^{\ell_k}] N_h &= \frac{1}{n-1} [t^{n-2} x_1^{\ell_1} \cdots x_k^{\ell_k}] \frac{x_h}{t} \prod_{i=1}^{k+1-h} F_{k,i}(t) \prod_{i=1}^{k+1-h} G_{k,i}(t)^{-1} \\
&\qquad \Big( \sum_{i=1}^{k+1-h} \frac{1}{G_{k,i}(t)} - \sum_{i=1}^{k+1-h} \frac{1}{F_{k,i}(t)} \Big) \\
&\qquad F_{k,1}(t)^{3n-3}  \Big( \prod_{i=2}^{\lceil k/2 \rceil} F_{k,i}(t) \prod_{i=1}^{\lfloor k/2 \rfloor} G_{k,i}(t)^{-1} \Big)^{4n-4} \\
&= \frac{1}{n-1} [t^{n-1} x_1^{\ell_1} \cdots x_h^{\ell_h-1} \cdots x_k^{\ell_k}] \Big( \prod_{i=k+2-h}^{\lceil k/2 \rceil} F_{k,i}(t) \prod_{i=k+2-h}^{\lfloor k/2 \rfloor} G_{k,i}(t)^{-1} \Big)^{4n-4} \\
&\qquad F_{k,1}(t)^{3n-2} \Big( \prod_{i=2}^{k+1-h} F_{k,i}(t) \prod_{i=1}^{k+1-h} G_{k,i}(t)^{-1} \Big)^{4n-3} \\
&\qquad \Big( \sum_{i=1}^{k+1-h} \frac{1}{G_{k,i}(t)} - \sum_{i=1}^{k+1-h} \frac{1}{F_{k,i}(t)} \Big).
\end{align*}
As before, we drop the variable $t$ now and write $F_{k,i}$ and $G_{k,i}$ instead of $F_{k,i}(1)$ and $G_{k,i}(1)$. The appropriate split in this case is
\begin{align*}
[z^n x_1^{\ell_1} \cdots x_k^{\ell_k}] N_h &= \frac{1}{n-1} [x_1^{\ell_1} \cdots x_h^{\ell_h-1} \cdots x_k^{\ell_k}] \Big( \prod_{i=k+2-h}^{\lceil k/2 \rceil} F_{k,i} \prod_{i=k+2-h}^{\lfloor k/2 \rfloor} G_{k,i}^{-1} \Big)^{4n-4} \\
&\qquad F_{k,1}^{3n-2} \Big( \prod_{i=2}^{k+1-h} F_{k,i} \prod_{i=1}^{k+1-h} G_{k,i}^{-1} \Big)^{4n-3} \\
&\qquad \Big( \sum_{i=1}^{k+1-h} \frac{1}{G_{k,i}} - \sum_{i=2}^{k+1-h} \frac{1}{F_{k,i}} - \frac{3n-2}{(4n-3)F_{k,1}} \Big) \\
&\quad- \frac{1}{n-1} [x_1^{\ell_1} \cdots x_h^{\ell_h-1} \cdots x_k^{\ell_k}] \Big( \prod_{i=k+2-h}^{\lceil k/2 \rceil} F_{k,i} \prod_{i=k+2-h}^{\lfloor k/2 \rfloor} G_{k,i}^{-1} \Big)^{4n-4} \\
&\qquad F_{k,1}^{3n-2} \Big( \prod_{i=2}^{k+1-h} F_{k,i} \prod_{i=1}^{k+1-h} G_{k,i}^{-1} \Big)^{4n-3} \cdot \frac{n-1}{(4n-3)F_{k,1}}\\
&= \frac{\ell_h}{(n-1)(4n-3)} [x_1^{\ell_1} \cdots x_k^{\ell_k}] \Big( \prod_{i=k+2-h}^{\lceil k/2 \rceil} F_{k,i} \prod_{i=k+2-h}^{\lfloor k/2 \rfloor} G_{k,i}^{-1} \Big)^{4n-4} \\
&\qquad F_{k,1}^{3n-2} \Big( \prod_{i=2}^{k+1-h} F_{k,i} \prod_{i=1}^{k+1-h} G_{k,i}^{-1} \Big)^{4n-3} \\
&\quad- \frac{1}{4n-3} [x_1^{\ell_1} \cdots x_h^{\ell_h-1} \cdots x_k^{\ell_k}] \Big( \prod_{i=k+2-h}^{\lceil k/2 \rceil} F_{k,i} \prod_{i=k+2-h}^{\lfloor k/2 \rfloor} G_{k,i}^{-1} \Big)^{4n-4} \\
&\qquad F_{k,1}^{3n-3} \Big( \prod_{i=2}^{k+1-h} F_{k,i} \prod_{i=1}^{k+1-h} G_{k,i}^{-1} \Big)^{4n-3}.
\end{align*}
Once again, we can now extract coefficients from both products following the order of variables $x_1,x_k,x_2,x_{k-1},\ldots$.

\medskip

Finally, we consider the generating function for all $k$-noncrossing trees, which is
$$N_1+N_2+ \cdots + N_k = \frac{x_1 z}{N_1} \Big(1 - \frac{x_1 z}{N_1}\Big) = \sqrt{\frac{z F_{k,1}(B)}{B}} - \frac{z F_{k,1}(B)}{B}$$
in view of~\eqref{eq:mainequation_nc} (for $r=1$) and~\eqref{eq:subst_nc}. So we can now apply the Lagrange-B\"urmann formula to the functions $f_1(t) = \sqrt{\frac{F_{k,1}(t)}{t}}$ (extracting the coefficient of $z^{n-1/2}$) and $f_2(t) = \frac{F_{k,1}(t)}{t}$ (extracting the coefficient of $z^{n-1}$), again with the same function $\Phi$ as before.
\end{proof}

As in the previous section, we can now derive a number of corollaries.

\begin{cor}\label{cor:total_nc}
For every integer $n > 1$, the total number of $k$-noncrossing trees with $n$ vertices is
$$\frac{1}{n-1} \binom{(2k+1)(n-1)}{n} - \frac{1}{2n-1} \binom{(2k+1)n-k-1}{n}.$$
The number of $k$-noncrossing trees with $n$ vertices whose root is labelled $h$ is
$$\frac{k+1-h}{2kn-k-h+1} \binom{(2k+1)n-k-h-1}{n-1}.$$
\end{cor}

\begin{proof}
We follow the lines of the proof of Corollary~\ref{cor:total}. Setting $x_1=x_2=\cdots = x_k = 1$, recall that we have
$$F_{k,i}(t) = \begin{cases} 1 & k \text{ even,} \\ 1 + t & k \text{ odd,} \end{cases}\quad \text{and} \quad G_{k,i}(t) = \begin{cases} 1 - t & k \text{ even,} \\ 1 & k \text{ odd.} \end{cases}$$
We show the calculations in the case that $k$ is even (the other case is similar once again), where we obtain
$$N_h = \sqrt{zB} (1-B)^{h-1}$$
and
$$N_1 + N_2 + \cdots + N_k = \frac{z}{N_1}\Big(1 - \frac{z}{N_1}\Big) = \sqrt{\frac{z}{B}} - \frac{z}{B},$$
where $B$ satisfies the implicit equation
$$B = z (1-B)^{-2k}.$$
We apply the Lagrange-B\"urmann formula to find that
\begin{align*}
[z^n] N_h &= [z^{n-1/2}] \sqrt{B} (1-B)^{h-1} \\
&= \frac{1}{n-\frac12} [t^{n-3/2}] \frac{1}{2\sqrt{t}} (1-(2h-1)t) (1-t)^{h-2} (1-t)^{-2k(n-1/2)}\\
&= \frac{1}{2n-1} [t^{n-1}] (1-(2h-1)t) (1-t)^{-2kn+k+h-2} \\
&= \frac{1}{2n-1} \binom{(2k+1)n-k-h}{n-1} - \frac{2h-1}{2n-1} \binom{(2k+1)n-k-h-1}{n-2} \\
&= \frac{k+1-h}{2kn-k-h+1} \binom{(2k+1)n-k-h-1}{n-1}
\end{align*}
and similarly
\begin{align*}
[z^n] (N_1 + N_2 + \cdots + N_k) &= [z^{n-1/2}] B^{-1/2} - [z^{n-1}] B^{-1} \\
&= \frac{1}{n-\frac12} [t^{n-3/2}] \Big( - \frac12 t^{-3/2} \Big) (1-t)^{-2k(n-1/2)} \\
&\quad- \frac{1}{n-1} [t^{n-2}] (-t^{-2}) (1-t)^{-2k(n-1)} \\
&= \frac{1}{n-1} [t^n] (1-t)^{-2kn+2k} - \frac{1}{2n-1} [t^n] (1-t)^{-2kn+k} \\
&= \frac{1}{n-1} \binom{(2k+1)(n-1)}{n} - \frac{1}{2n-1} \binom{(2k+1)n-k-1}{n}.
\end{align*}
\end{proof}
Next, we count $k$-noncrossing trees by the number of occurrences of a single label.

\begin{cor}\label{cor:single_label_nc}
For every $n > 1$, the total number of $k$-noncrossing trees with $n$ vertices of which $\ell$ are labelled $h$ is equal to
\begin{align*}
\frac{1}{n-1} &\sum_{r=0}^{n-\ell} \binom{4(h-1)(n-1)+r-1}{r} \binom{3n-3-r}{\ell} \binom{2(k+1-2h)(n-1)}{n-r-\ell}\\&-\frac{1}{2n-1} \sum_{r=0}^{n-\ell} \binom{2(h-1)(2n-1)+r-1}{r} \binom{3n-2-r}{\ell} \binom{(k+1-2h)(2n-1)}{n-r-\ell}
\end{align*}
if $h \leq \lceil k/2 \rceil$, and equal to
\begin{align*}
\frac{1}{n-1} &\sum_{r=0}^{n-\ell} \binom{(4h-4h+3)(n-1)}{r} \binom{n+r+\ell-2}{\ell} \binom{(4h-2k-3)(n-1)-r-\ell}{n-r-\ell}\\&-\frac{1}{2n-1} \sum_{r=0}^{n-\ell} \binom{2(k+1-h)(2n-1)-n}{r} \binom{n+r+\ell-1}{\ell} \binom{(2h-k-2)(2n-1)+n-r-\ell}{n-r-\ell}
\end{align*}
otherwise.
\end{cor}
\begin{proof}
Let us give the proof in the case that $k$ is odd and $h \leq \lceil k/2 \rceil = \frac{k+1}{2}$. The other cases are similar. Set all $x_i$ except $x_h$ equal to $1$. As noted in the proof of Corollary~\ref{cor:single_label}, we have
$$F_{k,i}(t) = \begin{cases} 1+x_ht & i \leq h, \\ 1+t & i > h,\end{cases} \qquad \text{and} \qquad G_{k,i}(t)  = \begin{cases} 1 + (x_h-1)t & i < h, \\ 1 & i \geq h. \end{cases}$$
Now we have to determine
\begin{equation}\label{eq:xh_extraction}
[z^n x_h^{\ell}] (N_1 + N_2 + \cdots + N_k)=[z^n x_h^{\ell}]\left(\sqrt{\frac{zF_{k,1}(B)}{B}}-\frac{zF_{k,1}(B)}{B}\right).
\end{equation}
We note that \[\frac{\partial}{\partial t}\frac{F_{k,1}(t)}{t}=-\frac{1}{t^2}\] and 
\[\frac{\partial}{\partial t}\sqrt{\frac{F_{k,1}(t)}{t}}=\frac{1}{2}\Big(\frac{F_{k,1}(t)}{t}\Big)^{-1/2}\frac{\partial}{\partial t}\frac{F_{k,1}(t)}{t}=-\frac{1}{2t^{3/2}}F_{k,1}(t)^{-1/2}.\]
Using the Lagrange-B\"urmann formula once again, we find that

\begin{align*}
[z^n x_h^{\ell}] &\sqrt{\frac{zF_{k,1}(B)}{B}} \\
&=[z^{n-1/2} x_h^{\ell}] \sqrt{\frac{F_{k,1}(B)}{B}} \\
&= \frac{1}{n-\frac12} [t^{n-3/2}x_h^{\ell}] \Big( -\frac{1}{2t^{3/2}}F_{k,1}(t)^{-1/2} \Big) \Big( F_{k,1}(t)^3 \prod_{i=2}^{\lceil k/2 \rceil} F_{k,i}(t)^4 \prod_{i=1}^{\lfloor k/2 \rfloor} G_{k,i}(t)^{-4} \Big)^{n-1/2} \\
&= -\frac{1}{2n-1} [t^n x_h^{\ell}] (1+x_h t)^{-1/2}(1+x_h t)^{3(n-1/2)}(1+x_h t)^{4(h-1)(n-1/2)} \\
&\qquad(1+t)^{2(k+1-2h)(n-1/2)} (1+(x_h-1)t)^{-4(h-1)(n-1/2)}\\
&=-\frac{1}{2n-1}[t^n x_h^{\ell}] (1+x_h t)^{3n-2}(1+x_h t)^{2(h-1)(2n-1)}(1+t)^{(k+1-2h)(2n-1)}\\
&\qquad(1+(x_h-1)t)^{-2(h-1)(2n-1)}.
\end{align*}
Now we extract the coefficient as follows:
\begin{align*}
[z^n x_h^{\ell}] &\sqrt{\frac{zF_{k,1}(B)}{B}}\\
&= -\frac{1}{2n-1} [t^n x_h^{\ell}] \Big( 1 - \frac{t}{1+x_h t}\Big)^{-2(h-1)(2n-1)} 
(1+x_h t)^{3n-2} (1+t)^{(k+1-2h)(2n-1)} \\
&= -\frac{1}{2n-1} [t^n x_h^{\ell}]  \sum_{r \geq 0} \binom{2(h-1)(2n-1)+r-1}{r} t^r (1+x_ht)^{3n-2-r} (1+t)^{(k+1-2h)(2n-1)} \\
&= -\frac{1}{2n-1} \sum_{r \geq 0} \binom{2(h-1)(2n-1)+r-1}{r} [t^{n-r} x_h^{\ell}] (1+x_ht)^{3n-2-r} (1+t)^{(k+1-2h)(2n-1)} \\
&= -\frac{1}{2n-1} \sum_{r \geq 0} \binom{2(h-1)(2n-1)+r-1}{r} \binom{3n-2-r}{\ell} [t^{n-r-\ell}] (1+t)^{(k+1-2h)(2n-1)} \\
&= -\frac{1}{2n-1} \sum_{r=0}^{n-\ell} \binom{2(h-1)(2n-1)+r-1}{r} \binom{3n-2-r}{\ell} \binom{(k+1-2h)(2n-1)}{n-r-\ell}.
\end{align*}
Similarly, we obtain $[z^n x_h^{\ell}] \frac{zF_{k,1}(B)}{B}.$ We have

\begin{align*}
[z^n x_h^{\ell}] &\frac{zF_{k,1}(B)}{B} =[z^{n-1} x_h^{\ell}] \frac{F_{k,1}(B)}{B} \\
&= \frac{1}{n-1} [t^{n-2} x_h^{\ell}] \Big( - \frac{1}{t^2} \Big) \Big( F_{k,1}(t)^3 \prod_{i=2}^{\lceil k/2 \rceil} F_{k,i}(t)^4 \prod_{i=1}^{\lfloor k/2 \rfloor} G_{k,i}(t)^{-4} \Big)^{n-1} \\
&=- \frac{1}{n-1}[t^n x_h^{\ell}] (1+x_h t)^{3n-3}(1+x_h t)^{4(h-1)(n-1)} \\
&\qquad(1+t)^{2(k+1-2h)(n-1)} (1+(x_h-1)t)^{-4(h-1)(n-1)}\\
&=- \frac{1}{n-1} [t^n x_h^{\ell}] \Big( 1 - \frac{t}{1+x_h t}\Big)^{-4(h-1)(n-1)} 
(1+x_h t)^{3n-3} (1+t)^{2(k+1-2h)(n-1)} \\
&=- \frac{1}{n-1} [t^n x_h^{\ell}]  \sum_{r \geq 0} \binom{4(h-1)(n-1)+r-1}{r} t^r (1+x_ht)^{3n-3-r} (1+t)^{2(k+1-2h)(n-1)} \\
&=- \frac{1}{n-1} \sum_{r \geq 0} \binom{4(h-1)(n-1)+r-1}{r} [t^{n-r} x_h^{\ell}] (1+x_ht)^{3n-3-r} (1+t)^{2(k+1-2h)(n-1)} \\
&=- \frac{1}{n-1} \sum_{r \geq 0} \binom{4(h-1)(n-1)+r-1}{r} \binom{3n-3-r}{\ell} [t^{n-r-\ell}] (1+t)^{2(k+1-2h)(n-1)} \\
&=- \frac{1}{n-1} \sum_{r=0}^{n-\ell} \binom{4(h-1)(n-1)+r-1}{r} \binom{3n-3-r}{\ell} \binom{2(k+1-2h)(n-1)}{n-r-\ell}.
\end{align*}
Now, we combine the two by means of~\eqref{eq:xh_extraction},
and the result follows.
\end{proof}
\begin{cor}\label{cor:average_nc}
For every $n > 1$, the average number of vertices labelled $h$ in $k$-noncrossing trees with $n$ vertices is
$$\frac{n}{(2k+1)n-(k+1)} \Big( 3n-2 - \frac{2(h-1)(n-1)}{k} + \frac{2(k+1-2h)}{(2k+1)\big(2 - \frac{(2kn+n-2k)^{\overline{k}}}{(2kn+1-2k)^{\overline{k}}}\big)} \Big),$$
where $m^{\overline{k}}=m(m+1)\cdots(m+k-1)$ is the rising factorial. Asymptotically, this is equal to $\frac{3k+2-2h}{k(2k+1)} n + \frac{k+1-2h}{(2 k + 1)^2( 2 (\frac{2k}{2k+1})^{k}-1)} + O(1/n)$.
\end{cor}

\begin{proof}
We only consider the case that $k$ is odd and $h \leq \lceil k/2 \rceil = \frac{k+1}{2}$, as in the previous proof. As in the proof of Corollary~\ref{cor:average}, instead of extracting coefficients, we take the derivative with respect to $x_h$ and plug in $x_h=1$ in order to determine the total number of vertices labelled $h$ in all $k$-noncrossing trees. All other variables $x_i$ are immediately taken to be $1$. This results in
\begin{align}\label{cor_single_nc_proof1}
\nonumber &[z^n] \frac{\partial}{\partial x_h} (N_1+N_2+\cdots + N_k) \Big|_{x_1=\cdots=x_k = 1} \\
\nonumber&= \frac{1}{n-1} [t^n] \frac{\partial}{\partial x_h}(1+x_h t)^{3n-3+4(h-1)(n-1)} (1+t)^{2(k+1-2h)(n-1)} (1+(x_h-1)t)^{-4(h-1)(n-1)} \Big|_{x_h=1}\\
\nonumber&\quad -\frac{1}{2n-1} [t^n] \frac{\partial}{\partial x_h}(1+x_h t)^{3n-2+2(h-1)(2n-1)} (1+t)^{(k+1-2h)(2n-1)} (1+(x_h-1)t)^{-2(h-1)(2n-1)} \Big|_{x_h=1} \\
\nonumber&= [t^n] t\big(3-4(h-1)t \big) (1+t)^{(2k+1)(n-1)-1}-[t^n]t\left(\frac{3n-2}{2n-1}-2(h-1)t \right) (1+t)^{(2k+1)n-k-2} \\
\nonumber&= 3 \binom{(2k+1)(n-1)-1}{n-1} - 4(h-1) \binom{(2k+1)(n-1)-1}{n-2}- \frac{3n-2}{2n-1} \binom{(2k+1)n-k-2}{n-1}\\
\nonumber&\quad + 2(h-1) \binom{(2k+1)n-k-2}{n-2}\\
\nonumber&=2(3k-2h+2)\binom{(2k+1)n-2k-2}{n-2}-\frac{(3nk-2k-2(h-1)(n-1))}{n-1}\binom{(2k+1)n-k-2}{n-2}.
\end{align}
Dividing by the total number of $k$-noncrossing trees, we obtain the stated formula after a number of simplifications.
\end{proof}

As noted for $k$-plane trees, it is also possible to derive formulas for the variance of the number of vertices labelled $h$, as well as covariances of two different label counts. Moreover, one could also take the root label into account in Corollary~\ref{cor:single_label_nc} and Corollary~\ref{cor:average_nc}. However, the formulas for $k$-noncrossing trees are even more complicated than for $k$-plane trees.  Instead of stating the most general result (which would be rather lengthy), we only present the special case $k=2$.

\begin{cor}
Let $n > 1$. Variances and covariances of the number of vertices labelled $1,2$ respectively in $2$-noncrossing trees with $n$ vertices are given in the following table:
\begin{center}
\begin{tabular}{c|cc}
& $1$ & $2$  \\
\hline
$1$ & $\frac{3 (2 n-1) (4 n-3)(49 n^2-100 n+44)}{25(5n-6)(7n-5)^2}$ & $-\frac{3 (2 n-1) (4 n-3)(49 n^2-100 n+44)}{25(5n-6)(7n-5)^2}$ \\
$2$ & $-\frac{3 (2 n-1) (4 n-3)(49 n^2-100 n+44)}{25(5n-6)(7n-5)^2}$ & $\frac{3 (2 n-1) (4 n-3)(49 n^2-100 n+44)}{25(5n-6)(7n-5)^2}$  \\
\end{tabular}
\end{center}
\end{cor}

\begin{proof}
We recall from the proof of Theorem~\ref{thm:main2} that
$$[z^nx_1^{\ell_1}x_2^{\ell_2}] (N_1+N_2) = \frac{1}{n-1} [t^nx_1^{\ell_1}x_2^{\ell_2}] \Phi(t)^{n-1}-\frac{1}{2n-1} [t^nx_1^{\ell_1}x_2^{\ell_2}] F_{2,1}(t)^{-1/2}\Phi(t)^{n-1/2},$$
where 
$$F_{2,1}(t) = 1+(x_1-x_2)t$$ and 
$$\Phi(t) = (1+(x_1-x_2)t)^3(1-x_2t)^{-4}$$
in the special case $k=2$. Again as in Corollary \ref{variance_plane}, for us to compute the variance of the number of vertices labelled $h$, we need to first compute the second moment, which is
\begin{equation}\label{eq:second_moment_nc}
\frac{[z^n] \frac{\partial^2}{\partial x_h^2} (N_1+N_2) |_{x_1=x_2= 1} + [z^n] \frac{\partial}{\partial x_h} (N_1+N_2) |_{x_1=x_2= 1}}{[z^n] (N_1+N_2) |_{x_1=x_2 = 1}},
\end{equation}
and then subtract the square of the mean. Likewise, the mixed moment of the number of vertices labelled $h$ and the number of vertices labelled $i$ is
$$\frac{[z^n] \frac{\partial^2}{\partial x_h \partial x_i} (N_1+N_2) |_{x_1=x_2= 1}}{[z^n] (N_1+N_2) |_{x_1=x_2 = 1}},$$
from which we subtract the product of the means to obtain the covariance.

\medskip

Again, we only show the calculations for the variance of the number of vertices labelled $1$ explicitly. The other entries follow automatically in this case, since the sum of the number of vertices labelled $1$ and the number of vertices labelled $2$ is deterministically equal to $n$. We get
\begin{align*}
[z^n] &\frac{\partial^2}{\partial x_1^2} (N_1+N_2) |_{x_1=x_2 = 1}\\
 &= \frac{1}{n-1} [t^n] 3(n-1)(3n-4)t^2 (1-t)^{-4(n-1)}- \frac{1}{2n-1} [t^n] (3n-2)(3n-3)t^2 (1-t)^{-(4n-2)}\\
&= 3(3n-4) [t^{n-2}](1-t)^{-4(n-1)}- \frac{(3n-2)(3n-3)}{2n-1} [t^{n-2}] (1-t)^{-(4n-2)}\\
&= 3(3n-4) \binom{5n-7}{n-2}-\frac{(3n-2)(3n-3)}{2n-1} \binom{5n-5}{n-2}.
\end{align*}
We already found earlier that
$$[z^n] \frac{\partial}{\partial x_1} (N_1+N_2) |_{x_1=x_2 = 1} = 3 \binom{5n-6}{n-1}-\frac{3n-2}{2n-1}\binom{5n-4}{n-1}$$
and
$$[z^n](N_1+N_2) |_{x_1=x_2 = 1} = \frac{1}{n-1} \binom{5n-5}{n}-\frac{1}{2n-1} \binom{5n-3}{n}.$$
Plugging everything into~\eqref{eq:second_moment_nc} and simplifying, we find a formula for the second moment and thus in turn for the variance.
\end{proof}

\begin{cor}
Let $n > 1$. The average number of vertices labelled $1,2$ respectively in $2$-noncrossing trees with $n$ vertices whose root is labelled $1,2$ respectively is given in the following table:
\begin{center}
\begin{tabular}{c|cc}
& $1$ & $2$  \\
\hline
root label $1$ & $\frac{3 n^{2} -n - 1}{5n-4}$ & $\frac{2n^2-3n+1}{5n-4}$  \\
root label $2$ & $\frac{3n-1}{5}$  & $\frac{2 n+1}{5}$  
\end{tabular}
\end{center}
\end{cor}

\begin{proof}
We recall from the proof of Theorem~\ref{thm:main2} that
$$N_1 = \frac{x_1 B(1-x_2B)^2}{(1 + (x_1-x_2)B)^2} ,\ N_2 = \frac{x_2 B(1-x_2B)^3}{(1 + (x_1-x_2)B)^2},$$
with
$$B = z (1+(x_1-x_2)B)^3(1-x_2B)^{-4}.$$
In order to determine the desired mean values, we need the coefficients of the partial derivatives $\frac{\partial}{\partial x_h} N_i |_{x_1=x_2 = 1}$. 
We will show the details of the calculations in one of the cases again: the number of vertices labelled $1$ in $2$-noncrossing trees whose root label is $1$.
Since
$$\frac{\partial}{ \partial t} \frac{x_1 t(1-x_2t)^2}{(1 + (x_1-x_2)t)^2} = \frac{x_1 (1-x_2t) (1-x_1t-2x_2t-x_1x_2 t^2+x_2^2 t^2)}{(1 + (x_1-x_2)t)^3},$$
we have
$$[z^n] N_1 = \frac{1}{n} [t^{n-1}] \frac{x_1 (1-x_2t) (1-x_1t-2x_2t-x_1x_2 t^2+x_2^2 t^2)}{(1 + (x_1-x_2)t)^3} (1+(x_1-x_2)t)^{3n}(1-x_2t)^{-4n}.$$
Now differentiate with respect to $x_1$ and set $x_1 = x_2 = 1$:
\begin{align*}
\frac{\partial}{\partial x_1} [z^n] N_1 \Big|_{x_1 = x_2= 1} &= \frac{1}{n} [t^{n-1}] \left(1+(3n-7)t-(9n-8)t^2\right) (1-t)^{1-4n} \\
&= \frac{1}{n} \left[ \binom{5n-3}{n-1} +(3n-7) \binom{5n-4}{n-2}-(9n-8)\binom{5n-5}{n-3}\right] \\
&= \frac{2(3n^2-n-1)}{(n-1)(n-2)} \binom{5n-5}{n-3}.
\end{align*}
Dividing by the total number of $2$-noncrossing trees with $n$ vertices and root label $1$, which is $\frac{1}{2n-1} \binom{5n-4}{n-1}$, we obtain the mean number of vertices labelled $1$.
 \end{proof}

%
%
%

\end{document}